\documentclass[preprint,12pt]{elsarticle}
\usepackage{amssymb}
\usepackage{amsmath}
\usepackage{amsthm}
\usepackage[capitalize]{cleveref}
\usepackage{blkarray}
\graphicspath{{Fig/}}
\usepackage{geometry}

\theoremstyle{thmstyletwo}%
\newtheorem{theorem}{Theorem}[section]
\newtheorem{example}{Example}[section]%
\newtheorem{remark}{Remark}[section]%

\numberwithin{equation}{section}
\newtheorem{lemma}{Lemma}[section]

\usepackage{amssymb}
\usepackage{amsmath}

\begin{document}

\begin{frontmatter}



\title{Generalized Singular Value Decompositions of Dual Quaternion Matrix Triplets} 

 \author{Sitao Ling\fnref{label2}}
 \ead{sling@cumt.edu.cn}
\fntext[label2]{School of Mathematics, China University of Mining and Technology, Jiangsu Xuzhou 221116,  P. R. China} 

\author{Wenxuan Ma\fnref{label2}\corref{cor1}} 
\cortext[cor1]{Corresponding author.}
\ead{Marcoper2013@gmail.com}

\author{Musheng Wei\fnref{label3}} 
\ead{mwei@shnu.edu.cn}
\fntext[label3]{College of Mathematics and Science, Shanghai Normal University, Shanghai 200234, P. R. China}



\begin{abstract}
	In signal processing and identification, generalized singular value decomposition (GSVD), related to a sequence of matrices in product/quotient form are
	essential numerical linear algebra tools. On behalf of the growing demand for efficient processing of coupled rotation-translation signals in modern engineering, we introduce the restricted SVD of a dual quaternion matrix triplet $(\boldsymbol{A},\boldsymbol{B},\boldsymbol{C})$ with $\boldsymbol{A}\in {\bf \mathbb{DQ}}^{m \times n}$, $\boldsymbol{B} \in {\bf \mathbb{DQ}}^{m \times p}$, $\boldsymbol{C} \in {\bf \mathbb{DQ}}^{q\times n}$, and the product-product SVD of a dual quaternion matrix triplet $(\boldsymbol{A},\boldsymbol{B},\boldsymbol{C})$ with $\boldsymbol{A}\in {\bf \mathbb{DQ}}^{m \times n}$, $\boldsymbol{B} \in {\bf \mathbb{DQ}}^{n \times p}$, $\boldsymbol{C} \in {\bf \mathbb{DQ}}^{p\times q}$. The two types of GSVDs represent a sophisticated matrix factorization that accounts for a given dual quaternion matrix in conjunction with two additional dual quaternion matrices. The decompositions can be conceptualized as an adaptation of the standard SVD, where the distinctive feature lies in the application of distinct inner products to the row and column spaces. Two examples are outlined to illustrate the feasibility of the decompositions.
\end{abstract}


\end{frontmatter}
\section{Introduction}

Dual quaternions, an advanced mathematical framework that builds upon the foundation of dual numbers, were initially introduced by British mathematician William Kingdon Clifford in 1873. This sophisticated tool has become increasingly valuable for enhancing computational efficiency and improving the quality of graphic rendering and animation. Notably, in the context of 3D rotations and interpolations, dual quaternions provide a solution that avoids the gimbal lock issue, enhances numerical stability, and ensures smooth transition effects. They also play a crucial role in addressing hand-eye calibration challenges, as well as in applications such as simultaneous localization and mapping (SLAM) and spacecraft position tracking. The importance of dual quaternions and their matrix representations are pivotal in these applications, as supported by extensive literature \cite{QWL2023}. 

Dual quaternion matrices, as a powerful algebraic tool for modeling coupled rotation-translation transformations, have shown unique advantages in signal processing tasks involving 3D/multi-dimensional dynamic signals. As a result, structure-preserving algorithms for the singular value decomposition (SVD) of dual quaternion matrices emerge \cite{DingWL,WLZG}. However, many practical scenarios (e.g., constrained point cloud registration, multi-channel signal equalization) require handling matrix triplets (sets of three dual quaternion matrices) with specific constraints (e.g., structural restrictions, subspace constraints). To address this, the restricted singular value decomposition (RSVD) of dual quaternion matrix triplets emerges as a critical technique, enabling constrained decomposition of coupled transformations while preserving key signal properties. This paper focuses on two specialized matrix decomposition techniques for dual quaternion matrix triplets: RSVD and PPSVD (product-product singular value decomposition). We dive deeper into each, and how they interact in the context of dual quaternion matrices.

Generalized singular value decomposition (GSVD) stands as a cornerstone in numerical linear algebra, with far-reaching applications across various fields such as signal processing, low-rank approximation of partitioned matrices, generalized eigenvalue problems, discrete linear ill-posed problems, and information retrieval. For an in-depth understanding of GSVD, one can refer to a comprehensive list of literature \cite{S1, S4, S6, S11, S13, S14, S19, S20, S23, S37, S45}. Among the many forms of GSVD, quotient singular value decomposition (QSVD) and product singular value decomposition (PSVD) of matrix pairs are well-known. 
RSVD is a less well-known extension of the orthogonal singular value decomposition (OSVD), which is particularly noteworthy. 

The RSVD has applications in canonical correlation analysis in statistics and signal processing, generalized Gauss-Markov models \cite{Zha1990}, etc. In the context of real matrix approximation problems, Zha's constructive proof \cite{S45} served as an exemplar that computing the RSVD in a numerically robust fashion is not a straightforward task. He characterized the algorithm as computationally impractical due to the incorporation of potentially ill-conditioned transformation matrices at intermediate stages. Zha addressed the specified challenge through the derivation of an implicit Kogbetliantz algorithm, as referenced in \cite{R29}. 
Ian N. Zwaan \cite{Zwaan} introduced a new algorithm based on Zha's implicit Kogbetliantz algorithm, enhancing the efficiency and accuracy of computing the RSVD by introducing innovative numerically stable algorithm and reducing the need for rank determinations. Chu, De Lathauwer, and De Moor \cite{R8} introduced a QR decomposition-based technique that does not necessitate a numerically stable $2\times 2$ RSVD. An alternative algorithm for determining the restricted singular values (RSVs) was proposed by Drma$\rm \check{c}$ \cite{R15}, employing a combination of Jacobi-type iterative procedures and non-orthonormal transformations. Zhang et al \cite{Zhanglp} developed neural network models for finding approximations of the RSVD and provided the globally asymptotic stability analyses.  

Recently, we investigated randomized QSVD for quaternion matrix pairs having the same number of columns. Specifically, we designed stable randomized algorithms for computing partial QSVD components \cite{LingLLJ} and conducted error analysis.
He, Michael and Zeng \cite{S2} investigated the generalized singular value decompositions (T-QSVD and T-PSVD) for tensors based on the T-product, explored their structures and algorithms, and applied them to color image watermarking, demonstrating the advantages in simultaneously processing two watermarks and simplifying key management. 

In the realm of matrix approximation problems, directly forming the product of real matrices $A, B, C$ with compatible dimensions often leads to significant errors, which can severely impact final experimental outcomes. To address this matrix approximation issue without directly constructing the product $ABC$, the PPSVD emerges as a highly effective tool \cite{Zha}.

To the best of our knowledge, there is a notable scarcity of literature focusing on the RSVD of dual quaternion matrix triplets. In a recent work \cite{LingMW}, we made significant strides by investigating the QSVD and the PSVD of dual quaternion matrix pairs. These studies not only deepened our understanding of the structural properties of dual quaternion matrices but also laid a solid theoretical foundation for our current exploration.  

This paper is organized as follows. In \cref{section2} we introduce the fundamental knowledge of dual quaternion and dual quaternion matrices. In \cref{sec:QtypeSVD} we hereby present our recent research findings on the QSVD for dual quaternion matrix pairs to facilitate further investigations. \cref{sec:RDQSVD} and \cref{sec:PPSVD} constitute the core work of this paper, which about the framework of RSVD and PPSVD for dual quaternion triplets. \cref{sec:exm} presents two illustrative examples to demonstrate the aforementioned decompositions. \cref{sec:concl} concludes the paper with a summary of key findings and implications.


%
%

\section{Dual quaternions and dual quaternion matrices}\label{section2}
Denote the sets of real numbers, dual numbers, quaternions, and dual quaternions by $\mathbb{R}$, $\mathbb{D}$, $\mathbb{Q}$, and $\mathbb{DQ}$, respectively. A dual number $\mathsf{q}$ is expressed as $\mathsf{q} = q_{{st}} + q_{{in}} \epsilon$, where $q_{{st}}, q_{{in}} \in \mathbb{R}$ and $\epsilon^2 = 0$. If the standard part $q_{{st}}$ is nonzero, then $\mathsf{q}$ is termed \textit{appreciable}. The conjugate of $\mathsf{q}$, denoted by $\mathsf{q}^*$, is itself.

For a given real coefficient polynomial $F(x)$, when evaluated at a dual number, we have $F(a + b\epsilon) = F(a) + F'(a)b\epsilon$, where $F'$ represents the derivative of $F$. The coefficient of $\epsilon$ is sometimes referred to as the \textit{infinitesimal} part, which intuitively implies that $\epsilon$ is \textquotedblleft so small\textquotedblright \ that it squares to zero.

As introduced by Qi and Ling \cite{s15}, a total order can be defined for dual numbers. For dual numbers $\mathsf{p} = p_{{st}} + p_{{in}} \epsilon$ and $\mathsf{q} = q_{{st}} + q_{{in}} \epsilon$ in $\mathbb{D}$ with $p_{{st}}, q_{{st}}, p_{{in}}, q_{{in}} \in \mathbb{R}$, we say $\mathsf{p} > \mathsf{q}$ if $p_{{st}} > q_{{st}}$ or $p_{{st}} = q_{{st}}$ and $p_{{in}} > q_{{in}}$; $\mathsf{p} = \mathsf{q}$ if and only if $p_{{st}} = q_{{st}}$ and $p_{{in}} = q_{{in}}$. We define $\mathsf{q}$ as a positive dual number if $\mathsf{q} > 0$, and a nonnegative dual number if $\mathsf{q} \ge 0$.

Consider a quaternion $\mathbf{q} = q_0 + q_1 \mathbf{i} + q_2 \mathbf{j} + q_3 \mathbf{k}$, where $q_0, q_1, q_2, q_3 \in \mathbb{R}$, and $\mathbf{i}, \mathbf{j}, \mathbf{k}$ are imaginary units satisfying
\begin{equation*}
	\bf{i}^{2}=\bf{j}^{2}=\bf{k}^{2}=\bf{ijk}=-1, \  \   \bf{ij}=-\bf{ji}=\bf{k},  \  \   \bf{jk}=-\bf{kj}=\bf{i}, \  \  \bf{ki}=-\bf{ik}=\bf{j}.
\end{equation*} 
The conjugate of $\mathbf{q}$ is $\mathbf{q}^* = q_0 - q_1 \mathbf{i} - q_2 \mathbf{j} - q_3 \mathbf{k}$, and the magnitude of $\mathbf{q}$ is $\left| \mathbf{q} \right| = \sqrt{\mathbf{q} \mathbf{q}^*} = \sqrt{q_0^2 + q_1^2 + q_2^2 + q_3^2}$.

As an extension of the dual number, a dual quaternion $\boldsymbol{q}$ is of the form $\boldsymbol{q} = \mathbf{q}_{{st}} + \mathbf{q}_{{in}} \epsilon$, where $\mathbf{q}_{{st}}$ and $\mathbf{q}_{{in}}$ are quaternions, and $\epsilon$ is the dual unit with $\epsilon^2 = 0$. $\epsilon$ commutes with multiplication when it encounters real numbers, complex numbers, or quaternions. We refer to $\mathbf{q}_{{st}}$ as the standard part of $\boldsymbol{q}$, and $\mathbf{q}_{{in}}$ as the dual part of $\boldsymbol{q}$. If $\mathbf{q}_{{st}} \neq 0$, we say that $\boldsymbol{q}$ is appreciable; otherwise, $\boldsymbol{q}$ is infinitesimal. The conjugate of $\boldsymbol{q}$ is denoted as $\boldsymbol{q}^* = \mathbf{q}_{{st}}^* + \mathbf{q}_{{in}}^* \epsilon$. The dual quaternion vector inner product $\langle \cdot, \cdot \rangle$ in the dual quaternion ring is defined by $\langle \boldsymbol{u}, \boldsymbol{v} \rangle = \boldsymbol{v}^* \boldsymbol{u}$ for any $\boldsymbol{u}, \boldsymbol{v} \in \mathbb{DQ}^m$, from which the induced dual quaternion vector norm, a dual number, is defined by $\left| \boldsymbol{u} \right|_2 = \sqrt{\langle \boldsymbol{u}, \boldsymbol{u} \rangle}$.

Let the sets of $m \times n$ matrices over the quaternions and dual quaternions be denoted by $\mathbb{Q}^{m \times n}$ and $\mathbb{DQ}^{m \times n}$, respectively. A dual quaternion matrix $\boldsymbol{A} = (\boldsymbol{a}_{ij}) \in \mathbb{DQ}^{m \times n}$ can be expressed as $\boldsymbol{A} = \mathbf{A}_{{st}} + \mathbf{A}_{{in}} \epsilon$, where $\mathbf{A}_{{st}}, \mathbf{A}_{{in}} \in \mathbb{Q}^{m \times n}$ represent the standard and infinitesimal parts of $\boldsymbol{A}$, respectively. We define $\boldsymbol{A}$ as \textit{appreciable} if $\mathbf{A}_{{st}} \neq \mathbf{O}$; otherwise, $\boldsymbol{A}$ is considered \textit{infinitesimal}. 
The conjugate transpose of $\boldsymbol{A}$ is given by $\boldsymbol{A}^*=(\boldsymbol{a}_{ji}^*) = \mathbf{A}_{{st}}^* + \mathbf{A}_{{in}}^* \epsilon$, and it holds that $(\boldsymbol{AB})^* = \boldsymbol{B}^* \boldsymbol{A}^*$ for any $\boldsymbol{B} \in \mathbb{DQ}^{n \times p}$. For a square matrix $\boldsymbol{A} \in \mathbb{DQ}^{m \times m}$, we say $\boldsymbol{A}$ is \textit{nonsingular} if there exists a square matrix $\boldsymbol{B} \in \mathbb{DQ}^{m \times m}$ such that $\boldsymbol{AB} = \boldsymbol{BA} = I_m$, and the inverse of $\boldsymbol{A}$ is $\boldsymbol{A}^{-1} = \boldsymbol{B}$. If the square matrix $\boldsymbol{A} \in \mathbb{DQ}^{m \times m}$ satisfies $\boldsymbol{A}^* = \boldsymbol{A}$, then $\boldsymbol{A}$ is \textit{Hermitian}. Additionally, $\boldsymbol{A}$ is \textit{unitary} if $\boldsymbol{A}^* \boldsymbol{A} = \boldsymbol{A} \boldsymbol{A}^* = I_m$.

Based on the definition of dual quaternion matrices and their operational properties, we can derive the following lemma.
\begin{lemma}\label{DM1}
	For any  ${\bf M} \in {\bf \mathbb{Q}}^{m \times n}$, 
	the matrix \begin{equation}\label{eq:DM1}
		\begin{pmatrix}
			I_m&{\bf M}\epsilon\\
			-{\bf M}^{*}\epsilon&I_n\\
		\end{pmatrix}
	\end{equation}
	is a unitary matrix.
\end{lemma}
\begin{proof}
	According to the definition of unitary matrices, the lemma is evidently valid.
\end{proof}

\section{Quotient SVD of a dual quaternion matrix pair}\label{sec:QtypeSVD}

In order to introduce the GSVD of dual quaternion matrix triplets, it is imperative to have a thorough understanding of the QR, SVD and quotient SVD of dual quaternion matrices.
 \begin{lemma}[QR decomposition]\label{QRT}
 	Suppose that $\boldsymbol{A}\in {\bf \mathbb{DQ}}^{m \times n}$ has full column rank. Then there exists a unitary matrix $\boldsymbol{Q}\in {\bf \mathbb{DQ}}^{m \times m}$ and an upper trapezoidal dual matrix $\boldsymbol{R}\in {\bf \mathbb{DQ}}^{r \times n}$ such that
 	$$\boldsymbol{A}=
 	\boldsymbol{Q}\begin{pmatrix}
 		\boldsymbol{R}\\
 		0 
 	\end{pmatrix},$$
 	where the number of non-infinitesimal rows  of $\boldsymbol{R}$ provides the appreciable rank of $\boldsymbol{A}$, denoted by ${\rm Arank}(\boldsymbol{A})$.
 \end{lemma}
\begin{lemma}[\cite{s12}, Theorem 6.1, DQSVD]\label{lm:4.1}
	For any $\boldsymbol{A} \in {\bf \mathbb{DQ}}^{m \times n}$, there exists a unitary dual quaternion matrices $\boldsymbol{U} \in {\bf \mathbb{DQ}}^{m \times m}$ and $\boldsymbol{V} \in {\bf \mathbb{DQ}}^{n \times n}$ such that 
	$$\boldsymbol{U}^{*}\boldsymbol{AV}=\begin{pmatrix}
		{\sf \Sigma}_{t}&0\\
		0&0 
	\end{pmatrix},$$
	where ${\sf \Sigma}_{t}={\rm diag}({\sf \mu}_{1}, \ldots, {\sf \mu}_{r}, \ldots, {\sf \mu}_{t})\in {\bf \mathbb{D}}^{t \times t}$,  $r \le t \le \min \{m,n\}$, ${\sf \mu}_{1} \ge {\sf \mu}_{2} \ge \cdots \ge {\sf \mu}_{r}$ are positive appreciable dual numbers, and ${\sf \mu}_{r+1} \ge {\sf \mu}_{r+2} \ge \cdots \ge {\sf \mu}_{t}$ are positive infinitesimal dual numbers. Counting possible multiplicities of the diagonal entries, the form of ${\sf \Sigma}_{t}$ is unique. The appreciable rank of $\boldsymbol{A}$ is $r$, i.e., ${\rm Arank}(\boldsymbol{A})=r$, and the rank of $\boldsymbol{A}$ is $t$, i.e., ${\rm rank}(\boldsymbol{A})=t$.
\end{lemma}
In \cite{LingMW}, we introduced two quotient SVD of a dual quaternion matrix pair $(\boldsymbol{A},\boldsymbol{B})$, where $\boldsymbol{A}$ and $\boldsymbol{B}$ have the same number of columns.
\begin{lemma}[DQGSVD1]\label{thm:4.3}
	Let $\boldsymbol{C}=\begin{pmatrix}
		\boldsymbol{A}\\
		\boldsymbol{B}
	\end{pmatrix}\in {\bf \mathbb{DQ}}^{(m+p) \times n}$ with $\boldsymbol{A}\in {\bf \mathbb{DQ}}^{m \times n}$, $\boldsymbol{B}\in {\bf \mathbb{DQ}}^{p \times n}$, and ${\rm rank}(\boldsymbol{C})=k$, ${\rm Arank}(\boldsymbol{C})=t$. There exist two unitary dual quaternion matrices $\boldsymbol{U}\in {\bf \mathbb{DQ}}^{m \times m}, \boldsymbol{V}\in {\bf \mathbb{DQ}}^{p \times p}$ and a dual quaternion matrix $\boldsymbol{X}\in {\bf \mathbb{DQ}}^{n \times n}$ such that 
	\begin{equation*}
		\boldsymbol{A}=\boldsymbol{U}({\sf \Sigma}_{\boldsymbol{A}}, \  0)\boldsymbol{X}, \  \
		\boldsymbol{B}=\boldsymbol{V}({\sf \Sigma}_{\boldsymbol{B}}, \ 0)\boldsymbol{X},
	\end{equation*}
	where
	\begin{equation*}
		\begin{aligned}
			&{\sf \Sigma}_{\boldsymbol{A}}=
			\begin{blockarray}{ccccc}
				r & q & t_{1} & k-r-q-t_{1}  \\
				\begin{block}{(cccc)c}
					I&0  &0&0&r\\
					0& \ {\sf S}_{\boldsymbol{A}}  &0&0&q \\
					0&0&\Xi\epsilon&0&t_{1}\\
					0 &0 &0&0&m-r-q-t_{1}\\
				\end{block}
			\end{blockarray},\\ 
			&{\sf \Sigma}_{\boldsymbol{B}}=
			\begin{blockarray}{cccccc}
				r_1 &r-r_1 & q & t_{1} & k-r-q-t_{1}  \\
				\begin{block}{(ccccc)c}
					{\Sigma}\epsilon&0&0&0&0&r_2\\
					0&0&0&0&0&p+r-k-r_2\\
					0&0&\ {\sf S}_{\boldsymbol{B}}   &0&0&q \\
					0&0&0&I&0&t_1\\
					0&0&0&0&I& k-r-q-t_{1}\\
				\end{block}
			\end{blockarray}.
		\end{aligned}
	\end{equation*}
	$\Sigma$, $\Xi$ are real diagonal matrices, ${\sf S}_{\boldsymbol{A}}$ and ${\sf S}_{\boldsymbol{B}}$ are appreciable dual matrices and
	\begin{equation*}
		\begin{split}
			\Sigma&={\rm diag}{(\sigma_{1},\sigma_{2},\ldots,\sigma_{r_{1}})},\quad \sigma_{1}\ge\sigma_{2}\cdots\ge\sigma_{r_{1}}>0, \\
			\Xi&={\rm diag}(\xi_1,\xi_2,\ldots,\xi_{t_{1}}), \quad \xi_1\ge \xi_2\ge\cdots\ge \xi_{t_{1}}>0,\\
			{\sf S}_{\boldsymbol{A}} &={\rm diag}({\sf c}_{1}, {\sf c}_{2}, \ldots, {\sf c}_{q}),\quad 0<{\sf c}_{1}\le \cdots\le {\sf c}_{q}<1,  \\
			{\sf S}_{\boldsymbol{B}} &={\rm diag}({\sf s}_{1}, {\sf s}_{2}, \ldots, {\sf s}_{q}),\quad 1>{\sf s}_{1}\ge\cdots\ge {\sf s}_{q}>0,\\
			&\qquad {\sf c}_{i}^{2}+{\sf s}_{i}^2=1,\quad i=1, 2,\ldots, q.
		\end{split}
	\end{equation*}
	More precisely, there exist two unitary dual quaternion matrices $\hat{\boldsymbol{U}}\in {\bf \mathbb{DQ}}^{m \times m}, \hat{\boldsymbol{V}}\in {\bf \mathbb{DQ}}^{p \times p}$ and a nonsingular dual quaternion matrix $\hat{\boldsymbol{X}}\in {\bf \mathbb{DQ}}^{n \times n}$ such that
	\begin{equation*}
		\boldsymbol{A}=\hat{\boldsymbol{U}}(\hat{\sf \Sigma}_{\boldsymbol{A}},  \boldsymbol{N}_{\boldsymbol{A}}\epsilon,0)\hat{\boldsymbol{X}}, \  \
		\boldsymbol{B}=\hat{\boldsymbol{V}}(\hat{{\sf \Sigma}}_{\boldsymbol{B}},\boldsymbol{N}_{\boldsymbol{B}}\epsilon, 0)\hat{\boldsymbol{X}},
	\end{equation*}
	where the blocked dual quaternion matrix 
	$\begin{blockarray}{cc}
		s &  \\
		\begin{block}{(c)c}
			\boldsymbol{N}_{\boldsymbol{A}} &  m\\
			\boldsymbol{N}_{\boldsymbol{B}}  &  p  \\
		\end{block}
	\end{blockarray}$ 
	has orthonormal columns,
	
	\begin{equation*}
		\begin{aligned}
			&\hat{{\sf \Sigma}}_{\boldsymbol{A}}=
			\begin{blockarray}{ccccc}
				\hat{r} & \hat{q} & \hat{t} & t-\hat{r}-\hat{q}-\hat{t} \\
				\begin{block}{(cccc)c}
					I&0  &0&0&\hat{r}\\
					0& \ \hat{{\sf S}}_{\boldsymbol{A}}  &0&0&\hat{q} \\
					0&0&\hat{\Xi}\epsilon&0&\hat{t}\\
					0 &0 &0&0&m-\hat{r}-\hat{q}-\hat{t}\\
				\end{block}
			\end{blockarray},\\ 
			&\hat{{\sf \Sigma}}_{\boldsymbol{B}}=
			\begin{blockarray}{cccccc}
				\hat{r}_1 &\hat{r}-\hat{r}_1 & \hat{q} & \hat{t} & t-\hat{r}-\hat{q}-\hat{t}  \\
				\begin{block}{(ccccc)c}
					\hat{\Sigma}\epsilon&0&0&0&0&\hat{r}_2\\
					0&0&0&0&0&p+\hat{r}-t-\hat{r}_2\\
					0&0&\  \hat{{\sf S}}_{\boldsymbol{B}}   &0&0&\hat{q} \\
					0&0&0&I&0&\hat{t}\\
					0&0&0&0&I& t-\hat{r}-\hat{q}-\hat{t}\\
				\end{block}
			\end{blockarray}.
		\end{aligned}
	\end{equation*}
	$\hat{\Sigma}$, $\hat{\Xi}$ are real diagonal matrices, $\hat{{\sf S}}_{\boldsymbol{A}}$ and $\hat{{\sf S}}_{\boldsymbol{B}}$ are appreciable dual matrices and
	\begin{equation*}
		\begin{split}
			\hat{\Sigma}&={\rm diag}{(\hat{\sigma}_{1},\hat{\sigma}_{2},\ldots,\hat{\sigma}_{\hat{r}_{1}})},\quad \hat{\sigma}_{1}\ge \hat{\sigma}_{2}\ge\cdots\ge \hat{\sigma}_{\hat{r}_{1}}>0,\\
			\hat{\Xi}&={\rm diag}(\hat{\xi}_1,\hat{\xi}_2,\ldots\,\hat{\xi}_{\hat{t}}),\quad \hat{\xi}_1\ge \hat{\xi}_2\ge\cdots\ge \hat{\xi}_{\hat{t}}>0,\\
			\hat{{\sf S}}_{\boldsymbol{A}} &={\rm diag}(\hat{{\sf c}}_{1}, \hat{{\sf c}}_{2}, \ldots, \hat{{\sf c}}_{\hat{q}}),\quad 0<\hat{{\sf c}}_{1}\le \cdots\le \hat{{\sf c}}_{\hat{{q}}}<1,  \\
			\hat{{\sf S}}_{\boldsymbol{B}} &={\rm diag}(\hat{{\sf s}}_{1}, \hat{{\sf s}}_{2}, \ldots, \hat{{\sf s}}_{\hat{q}}),\quad 1>\hat{{\sf s}}_{1}\ge\cdots\ge \hat{{\sf s}}_{\hat{q}}>0,\\
			&\qquad \hat{{\sf c}}_{i}^{2}+\hat{{\sf s}}_{i}^2=1,\quad i=1,2,\ldots,\hat{q}.
		\end{split}
	\end{equation*}
\end{lemma}
\begin{lemma}[DQGSVD2]\label{DDQGSVD}
	Let	$\boldsymbol{C}=\begin{pmatrix}
		\boldsymbol{A}\\
		\boldsymbol{B}
	\end{pmatrix} \in {\bf \mathbb{DQ}}^{(m+p) \times n} 	$
	with $\boldsymbol{A} \in {\bf \mathbb{DQ}}^{m \times n}, \boldsymbol{B} \in {\bf \mathbb{DQ}}^{p \times n}$, and ${\rm Arank}(\boldsymbol{C})=t$. 
	There exist two unitary dual quaternion matrices $\boldsymbol{U}\in {\bf \mathbb{DQ}}^{m \times m}, \boldsymbol{V}\in {\bf \mathbb{DQ}}^{p \times p}$  and a dual quaternion matrix $\boldsymbol{X}\in {\bf \mathbb{DQ}}^{n \times n}$, such that 
	\begin{equation*}
		\boldsymbol{U}^{*}\boldsymbol{A}\boldsymbol{X}=
		({{\sf \Sigma}_{\boldsymbol{A}}}, \  0), \  \
		\boldsymbol{V}^{*}\boldsymbol{BX}=
		({\sf \Sigma}_{\boldsymbol{B}}, \  0),
	\end{equation*}
	where
	\begin{subequations}\label{sigABi}
		\begin{align*}
			{\sf \Sigma}_{\boldsymbol{A}}
			&= \begin{blockarray}{ccccc}
				r& q & l  & t-(r+q+l)&   \\
				\begin{block}{(cccc)c}
					I &0  &0 &0  & r  \\
					0&\ {\sf S}_{\boldsymbol{A}}  &0 &0  & q\\
					0&0&{\Xi}\epsilon&0  & l \\
					0&0&0&0  &  m-r-q-l   \\
				\end{block}
			\end{blockarray},  
			\end{align*}
		\end{subequations}
			\begin{subequations}
			\begin{align*}
			{{\sf \Sigma}}_{\boldsymbol{B}}&=
			\begin{blockarray}{ccccc}
				r_1& r-r_1 & q & t-(r+q)\\
				\begin{block}{(cccc)c}
					{\Sigma}\epsilon&0&0&0&r_2\\
					0&0&0&0&p-t+r-r_2\\
					0&0&{\sf S}_{\boldsymbol{B}}&0&q\\
					0&0&0&I&t-(r+q)\\
				\end{block}
			\end{blockarray}.
		\end{align*}
	\end{subequations}
	$\Sigma$, $\Xi$ are real diagonal matrices,  ${\sf S}_{\boldsymbol{A}}$ and ${\sf S}_{\boldsymbol{B}}$ are appreciable dual matrices and
	\begin{equation*}
		\begin{split}
			\Sigma&={\rm diag}{(\sigma_{1},\sigma_{2},\ldots,\sigma_{r_{1}})}, \  \   \sigma_{1}\ge\sigma_{2}\ge\cdots\ge\sigma_{r_{1}}>0,\\
			\Xi &={\rm diag}(\xi_1,\xi_2,\ldots,\xi_l),\ \xi_1\ge \xi_2\ge\cdots\ge \xi_l>0,   \\
			{\sf S}_{\boldsymbol{A}} &={\rm diag}({\sf c}_{1}, {\sf c}_{2}, \ldots, {\sf c}_{q}),\ 0<{\sf c}_{1}\le \cdots\le {\sf c}_{q}<1, \\
			{\sf S}_{\boldsymbol{B}} &={\rm diag}({\sf s}_{1}, {\sf s}_{2}, \ldots, {\sf s}_{q}),\ 1>{\sf s}_{1}\ge\cdots\ge{\sf s}_{q}>0, \\
			&\qquad {\sf c}_{i}^{2}+{\sf s}_{i}^2=1,\ i=1,2,\ldots,q.
		\end{split}
	\end{equation*}	
\end{lemma}

\section{Restricted SVD of a dual quaternion matrix triplet}\label{sec:RDQSVD}

The RSVD of the real matrix triplet $(A,B,L)$ is closed related to the regularized Gauss-Markov linear model \cite{Zha1990}
$$\min\{\|u\|_2^2+\lambda^2\|Lx\|_2^2\}, \quad  \text{s.t.} \quad Ax+Bu=b,$$
where $A\in \mathbb{R}^{m\times n}, B\in \mathbb{R}^{m\times q}$ and $L\in \mathbb{R}^{p\times n}$.  
With the aid of aformationed two types of DQGSVD of a dual quaternion matrix pair, we further study the restricted SVD of a dual quaternion matrix triplet $(\boldsymbol{A}, \boldsymbol{B},\boldsymbol{C})$, denoted by DQRSVD for short. The DQRSVD is an extension of the DQGSVD that can handle the decomposition of three dual quaternion matrices simultaneously. It has potential application prospects  such as multivariable system analysis, tensor decomposition, and signal processing involving multiple data sets.

\begin{theorem}[DQRSVD1]\label{DQRSVD111}
	Let $\boldsymbol{A} \in {\bf \mathbb{DQ}}^{m \times n}$, $\boldsymbol{B} \in ^{m \times p}$, and $\boldsymbol{C} \in {\bf \mathbb{DQ}}^{q \times n}$. There exist dual quaternion matrices $\boldsymbol{P}\in {\bf \mathbb{DQ}}^{m \times m}$, $\boldsymbol{Q}\in {\bf \mathbb{DQ}}^{n \times n}$, and unitary dual quaternion matrices $\boldsymbol{U}\in {\bf \mathbb{DQ}}^{p\times p}$, $\boldsymbol{V}\in {\bf \mathbb{DQ}}^{q\times q}$ such that	
	\begin{subequations}\label{RSVDresu1}
		\begin{align}
			\boldsymbol{PAQ}&=
			\begin{blockarray}{ccc}
				n-k_{1} &k_{1}\\ 
				\begin{block}{(cc)c}
					{\sf \Sigma}_{\boldsymbol{A}}& 0  &m_{1}  \\
					0&0_{\boldsymbol{A}}&m_{2}\\
				\end{block}
			\end{blockarray},    \  \
			\boldsymbol{PBU}=
			\begin{blockarray}{cc}
				p\\
				\begin{block}{(c)c}
					{\sf \Sigma}_{\boldsymbol{B}}&m-k_{2}  \\ 
					0_{\boldsymbol{B}}&k_{2}  \\
				\end{block}
			\end{blockarray},  \\
			\boldsymbol{VCQ}&=
			\begin{blockarray}{ccc}
				n-k_{1} &k_{1}\\ 
				\begin{block}{(cc)c}
					{\sf \Sigma}_{\boldsymbol{C}}&0_{\boldsymbol{C}}&q  \\ 
				\end{block}
			\end{blockarray},
		\end{align}
	\end{subequations}
	where $m_1+m_2=m$, and 
	\begin{subequations}\label{DQRSVD}
		\begin{align*}
			{\sf \Sigma}_{\boldsymbol{A}}&=\begin{blockarray}{cccccc}
				j+k+l &r&s&t&\hat{s}_{1}\\ 
				\begin{block}{(ccccc)c}
					I&0&0&0&0&j+k+l  \\ 
					0&I&0&0&0&r  \\
					0&0&{\sf S}_{\boldsymbol{A}}&0&0&s\\
					0&0&0&\widetilde{S}_{\boldsymbol{A}}\epsilon&0&t\\
					0&0&0&0&0&\hat{t}\\
				\end{block}
			\end{blockarray}, 
			\end{align*}
		\end{subequations}
	\begin{subequations}
	\begin{align*}
			{\sf \Sigma}_{\boldsymbol{B}} &=
			\begin{blockarray}{cccc}
				p-s-\hat{t} &s&t+\hat{t}\\ 
				\begin{block}{(ccc)c}
					\begin{pmatrix}
						I&0&0\\
						0&I\epsilon&0\\
						0&0&0
					\end{pmatrix}&0&0&j+k+l  \\ 
					\hat{\bf B}\epsilon&0&0&r\\
					0&{\sf S}_{\boldsymbol{B}}&0&s  \\
					0&0&I&t+\hat{t}\\
				\end{block}
			\end{blockarray}, \\
			{\sf \Sigma}_{\boldsymbol{C}}
			&=\begin{blockarray}{cccccc}
				j+k+l &r&s&t&\hat{s}_{1}\\ 
				\begin{block}{(ccccc)c}
					\hat{\bf C}_{1}\epsilon&0&0&0&0&q-r-s-t-\hat{s}_{1}  \\ 
					0&I&0&0&0&r\\
					0&0&{\sf S}_{\boldsymbol{C}}&0&0&s\\
					0&0&0&I&0&{t}\\
					0&0&0&0&I&\hat{s}_{1}\\
				\end{block}
			\end{blockarray}.
		\end{align*}
	\end{subequations}
	The submatrices $\hat{\bf B}\epsilon$, $\hat{\bf C}_1\epsilon$ are given by \eqref{subAX}, ${\widetilde S}_{\boldsymbol{A}}={\rm{diag}}(d_{i})$ is a real matrix,  ${\sf S}_{\boldsymbol{A}}={\rm{diag}}(\alpha_{i})$, ${\sf S}_{\boldsymbol{B}}={\rm{diag}}(\beta_{i})$, ${\sf S}_{\boldsymbol{C}}={\rm{diag}}(\gamma_{i})$ are dual matrices, and
	\begin{equation}\label{RSVDsq}
		\begin{split}
			\alpha_{i}^{2}+\beta_{i}^{2}+\gamma_{i}^{2}=1, \ \    i&=\lambda+1,\ldots,\lambda+s, \\
			\lambda&=j+k+l+r.
		\end{split}
	\end{equation}
	Furthermore, 
	\begin{equation}\label{RSVDEQ1}
		1>\alpha_{i}\ge \alpha_{i+1}>0,\quad0<\beta_{i}\leq\beta_{i+1}<1,\quad1>\gamma_{i}\ge \gamma_{i+1}>0
	\end{equation}
	and 
	\begin{equation}{\label{RSVDEQ2}}
		\frac{\alpha_{i}}{\beta_{i}\gamma_{i}}\ge\frac{\alpha_{i+1}}{\beta_{i+1}\gamma_{i+1}},\quad i=\lambda+1,\ldots,\lambda+s-1.
	\end{equation}
\end{theorem}

\begin{proof}
	The proof is performed in four steps. The transformations are represented to be the following recursive relations
	\begin{equation}\label{updateeq}
		\begin{split}
		\boldsymbol{A}^{(k+1)}&=\boldsymbol{P}^{(k)}\boldsymbol{A}^{(k)}\boldsymbol{Q}^{(k)},  \
		\boldsymbol{B}^{(k+1)}=\boldsymbol{P}^{(k)}\boldsymbol{B}^{(k)}\boldsymbol{U}^{(k)},   \\
		\boldsymbol{C}^{(k+1)}&=\boldsymbol{V}^{(k)}\boldsymbol{C}^{(k)}\boldsymbol{Q}^{(k)},
	\end{split}
	\end{equation}
	where $\boldsymbol{P}^{(k)}\in {\bf \mathbb{DQ}}^{m \times m}$ and $\boldsymbol{Q}^{(k)}\in {\bf \mathbb{DQ}}^{n \times n}$, $\boldsymbol{U}^{(k)}\in {\bf \mathbb{DQ}}^{p \times p}$ and $\boldsymbol{V}^{(k)}\in {\bf \mathbb{DQ}}^{q \times q}$ are unitary. In each step, we only specify $\boldsymbol{P}^{(k)}$, $\boldsymbol{Q}^{(k)}$, $\boldsymbol{U}^{(k)}$, $\boldsymbol{V}^{(k)}$ and the updated matrices $\boldsymbol{A}^{(k+1)}$, $\boldsymbol{B}^{(k+1)}$ and $\boldsymbol{C}^{(k+1)}$. We preset
	\begin{equation*}
		\boldsymbol{A}^{(1)}=\boldsymbol{A},\ \boldsymbol{B}^{(1)}=\boldsymbol{B},\  \boldsymbol{C}^{(1)}=\boldsymbol{C}.
	\end{equation*}
	
	Step 1. From \cref{DDQGSVD}, the DQGSVD2 of $\begin{pmatrix}
		\boldsymbol{A}^{(1)}\\
		\boldsymbol{C}^{(1)}
	\end{pmatrix}$ 
	has the form	
	
	\begin{align*}
		&\boldsymbol{U}_{1}\boldsymbol{A}^{(1)}\boldsymbol{Q}_{1} = ({\sf \Sigma}_{\boldsymbol{A}}^{(1)},0)	\\
		&=\small
		\begin{blockarray}{cccccc}
			j+k+l &r_1&r_2&r_3&n-(j+k+l+r_{1}+r_{2}+r_{3})\\ 
			\begin{block}{(cccc|c)c}
				I&0&0&0&0&j+k+l  \\ 
				0&{\sf S}_{\boldsymbol{A}}^{(1)}&0&0&0&r_1\\
				0&0&\widetilde{S}_{\boldsymbol{A}}^{(1)}\epsilon&0&0&r_2\\
				0&0&0&0&0&m-(j+k+l+r_1+r_2)\\
			\end{block}
		\end{blockarray},\\
		& \boldsymbol{V}_{1}\boldsymbol{C}^{(1)}\boldsymbol{Q}_{1} = ({\sf \Sigma}_{\boldsymbol{C}}^{(1)},0)\\
		&=\small
		\begin{blockarray}{cccccc}
			j+k+l &r_1&r_2&r_3&n-(j+k+l+r_{1}+r_{2}+r_{3})\\ 
			\begin{block}{(cccc|c)c}
				\begin{pmatrix}
					{\Sigma}_{c}^{(1)}\epsilon&0\\
					0&0
				\end{pmatrix}&0&0&0&0&q-r_1-r_2-r_3  \\ 
				0&{\sf S}_{\boldsymbol{C}}^{(1)}&0&0&0&r_1\\
				0&0&I&0&0&r_2\\
				0&0&0&I&0&r_3\\
			\end{block}
		\end{blockarray},	
	\end{align*}
	where $\boldsymbol{U}_{1}, \boldsymbol{V}_{1}$ are both unitary dual quaternion matrices, and $r_1+r_2+r_3=r+s+t_1+t_2$.
	
	Let
	\begin{equation*}
		\begin{split}
			\boldsymbol{P}^{(1)}&=\boldsymbol{U}_{1}\in \mathbb{DQ}^{m \times m},\quad\boldsymbol{Q}^{(1)}=\boldsymbol{Q}_{1}{\rm diag}(I_{j+k+l},(S_{\boldsymbol{C}}^{(1)})^{-1},I_{n-(j+k+l+r_{1})}),  \\
			\boldsymbol{U}^{(1)}&=I_p,\quad\boldsymbol{V}^{(1)}=\boldsymbol{V}_{1} \in \mathbb{DQ}^{q \times q}.
		\end{split}
	\end{equation*}
	Then the dual quaternion matrices $\boldsymbol{A}^{(1)}, \boldsymbol{B}^{(1)}, \boldsymbol{C}^{(1)}$ are updated to be
	\begin{align*}
		\boldsymbol{A}^{(2)}&=\boldsymbol{P}^{(1)}\boldsymbol{A}^{(1)}\boldsymbol{Q}^{(1)}\\
		&= \tiny
		\begin{blockarray}{cccccc}
			j+k+l &r_1&r_2&r_3&n-(j+k+l+r_{1}+r_{2}+r_{3})\\ 
			\begin{block}{(cccc|c)c}
				I&0&0&0&0&j+k+l  \\ 
				0&{\sf S}_{\boldsymbol{A}}^{(1)}({\sf S}_{\boldsymbol{C}}^{(1)})^{-1}&0&0&0&r_1\\
				0&0&\widetilde{S}_{\boldsymbol{A}}^{(1)}\epsilon&0&0&r_2\\
				0&0&0&0&0&m-(j+k+l+r_1+r_2)\\
			\end{block}
		\end{blockarray}, \\
		\boldsymbol{B}^{(2)}&=\boldsymbol{P}^{(1)}\boldsymbol{B}^{(1)}\boldsymbol{U}^{(1)}=
		\begin{blockarray}{cc}
			p&  \\
			\begin{block}{(c)c}
				\boldsymbol{B}_{1}^{(2)}&j+k+l  \\ 
				\boldsymbol{B}_{2}^{(2)}&m-(j+k+l)  \\
			\end{block}
		\end{blockarray}, 
			\end{align*}
			\begin{align*}
		\boldsymbol{C}^{(2)}&=\boldsymbol{V}^{(1)}\boldsymbol{C}^{(1)}\boldsymbol{Q}^{(1)}\\
		&= \small
		\begin{blockarray}{cccccc}
			j+k+l &r_1&r_2&r_3&n-(j+k+l+r_{1}+r_{2}+r_{3})\\ 
			\begin{block}{(cccc|c)c}
				\begin{pmatrix}
					{\Sigma}_{c}^{(1)}\epsilon&0\\
					0&0
				\end{pmatrix}&0&0&0&0&q-r_1-r_2-r_3  \\ 
				0&I&0&0&0&r_1\\
				0&0&I&0&0&r_2\\
				0&0&0&I&0&r_3\\
			\end{block}
		\end{blockarray}.
	\end{align*}
	
	Step 2. Notice that the following blocked matrix 
	\begin{equation}\label{submx}
		\begin{pmatrix}
			\begin{pmatrix}
				{\sf S}_{\boldsymbol{A}}^{(1)}({\sf S}_{\boldsymbol{C}}^{(1)})^{-1}&0\\
				0&\widetilde{S}_{\boldsymbol{A}}^{(1)}\epsilon\\
				0&0
			\end{pmatrix},&\boldsymbol{B}_{2}^{(2)}
		\end{pmatrix} \in\mathbb{DQ}^{(m-j-k-l)  \times (r_1+r_2+r_3+p)}
	\end{equation}
	is formed by the submatrices taking from $\boldsymbol{A}^{(2)}$ and $\boldsymbol{B}^{(2)}$.
	Applying \cref{DDQGSVD} once again but to the conjugate transpose of \eqref{submx}, there exist unitary dual quaternion matrices $\boldsymbol{U}_{2}\in\mathbb{DQ}^{p\times p}$, $\boldsymbol{Q}_{2}\in\mathbb{DQ}^{(r_1+r_2+r_3)\times (r_1+r_2+r_3)}$
	and a dual quaternion matrix $\boldsymbol{P}_{2} \in\mathbb{DQ}^{(m-j-k-l)  \times (m-j-k-l)}$ such that
	\begin{align*}
		\boldsymbol{P}_{2}\begin{pmatrix}
			{\sf S}_{\boldsymbol{A}}^{(1)}({\sf S}_{\boldsymbol{C}}^{(1)})^{-1}&0\\
			0&\widetilde{S}_{\boldsymbol{A}}^{(1)}\epsilon\\
			0&0
		\end{pmatrix}\boldsymbol{Q}_{2}=
		\begin{blockarray}{cccc}
			r &s&t\\ 
			\begin{block}{(ccc)c}
				I&0&0&r  \\ 
				0&{\sf S}_{\boldsymbol{A}}^{(2)}&0&s  \\
				0&0&\widetilde{S}_{\boldsymbol{A}}^{(2)}\epsilon&t\\
				0&0&0&\hat{t}\\
				0&0&0&t_{1}\\
			\end{block}
		\end{blockarray},\\
		\boldsymbol{P}_{2}\boldsymbol{B}_{2}^{(2)}\boldsymbol{U}_{2}=
		\begin{blockarray}{cccc}
			p-s-t&s&t+\hat{t}\\ 
			\begin{block}{(ccc)c}
				\begin{pmatrix}
					{\Sigma}_{b}^{(2)}\epsilon&0\\
					0&0
				\end{pmatrix}&0&0&r  \\ 
				0&{\sf S}_{\boldsymbol{B}}^{(2)}&0&s  \\
				0&0&I&t+\hat{t}\\
				0&0&0&t_{1}\\
			\end{block}
		\end{blockarray},
	\end{align*}
	where ${\sf S}_{\boldsymbol{A}}^{(2)}={\rm{diag}}(\theta_{1},\ldots,\theta_{s})$, 
	${\sf S}_{\boldsymbol{B}}^{(2)}={\rm{diag}}(\delta_{1},\ldots,\delta_{s})$ and $\theta_{i}^{2}+\delta_{i}^{2}=1$, $1>\theta_{1}\ge\theta_{2}\ge\cdots \ge \theta_{s}>0$, $0<\delta_{1}\leq\delta_{2}\leq\cdots\leq\delta_{s}<1$. 
	Notice that this is the conjugate transpose of the results after performing DQGSVD2 in \cref{DDQGSVD}.
	
	Let
	\begin{equation}\label{eq:PQUV2}
		\begin{split}
			\boldsymbol{P}^{(2)}&={\rm diag}(I_{j+k+l},\boldsymbol{P}_{2}),\quad \boldsymbol{Q}^{(2)}={\rm diag}(I_{j+k+l},\boldsymbol{Q}_{2},I_{n-j-k-l-r-s-t}),  \\
			\boldsymbol{U}^{(2)}&=\boldsymbol{U}_{2}\in \mathbb{DQ}^{p\times p}, \quad
			\boldsymbol{V}^{(2)}={\rm diag}(I_{j+k+l},\boldsymbol{Q}_{2}^{*}).
		\end{split}
	\end{equation}
	Then the dual quaternion matrices $\boldsymbol{A}^{(2)}, \boldsymbol{B}^{(2)}, \boldsymbol{C}^{(2)}$ are updated to be
	\begin{align*}
		\boldsymbol{A}^{(3)}&=\boldsymbol{P}^{(2)}\boldsymbol{A}^{(2)}\boldsymbol{Q}^{(2)}
		=\begin{blockarray}{cccccc}
			j+k+l &r&s&t&h\\ 
			\begin{block}{(cccc|c)c}
				I&0&0&0&0&j+k+l  \\ 
				0&I&0&0&0&r\\
				0&0&{\sf S}_{\boldsymbol{A}}^{(2)}&0&0&s\\
				0&0&0&\widetilde{S}_{\boldsymbol{A}}^{(2)}\epsilon&0&t\\
				0&0&0&0&0&t_{1}+\hat{t}\\
			\end{block}
		\end{blockarray},
	\end{align*}
	\begin{align*}
		\boldsymbol{B}^{(3)}&=\boldsymbol{P}^{(2)}\boldsymbol{B}^{(2)}\boldsymbol{U}^{(2)}
		=
		\begin{blockarray}{cccc}
			p-s-t-\hat{t}&s&t+\hat{t}\\ 
			\begin{block}{(ccc)c}
				\boldsymbol{B}_{1}^{(3)}&\boldsymbol{B}_{2}^{(3)}&\boldsymbol{B}_{3}^{(3)}&j+k+l\\
				\begin{pmatrix}
					\Sigma_{b}^{(2)}\epsilon&0\\
					0&0
				\end{pmatrix}&0&0&r  \\ 
				0&{\sf S}_{\boldsymbol{B}}^{(2)}&0&s  \\
				0&0&I&t+\hat{t}\\
				0&0&0&t_{1}\\
			\end{block}
		\end{blockarray},  \\
		\boldsymbol{C}^{(3)}&=\boldsymbol{C}^{(2)}.
	\end{align*}

	Step 3. In this step, we perform blocked elementary transformations to update $\boldsymbol{A}^{(3)}, \boldsymbol{B}^{(3)}$ and $\boldsymbol{C}^{(3)}$, by the aid of the following defined  blocked elementary dual quaternion matrices
	\begin{align*}
		\boldsymbol{P}^{(3)}&=
		\begin{blockarray}{cccccc}
			j+k+l&r&s&t+\hat{t}&t_{1}\\ 
			\begin{block}{(ccccc)c}
				I&0&-\boldsymbol{B}_{2}^{(3)}({\sf S}_{\boldsymbol{B}}^{(2)})^{-1}&-\boldsymbol{B}_{3}^{(3)}&0&j+k+l\\
				0&I&0&0&0&r\\
				0&0&I&0&0&s\\
				0&0&0&I&0&t+\hat{t}\\
				0&0&0&0&I&t_{1}\\
			\end{block}
		\end{blockarray},  \\
		\boldsymbol{Q}^{(3)}&=
		\begin{blockarray}{cccccc}
			j+k+l&r&s&t+\hat{t}&h\\ 
			\begin{block}{(ccccc)c}
				I&0&\boldsymbol{B}_{2}^{(3)}({\sf S}_{\boldsymbol{B}}^{(2)})^{-1}{\sf S}_{\boldsymbol{A}}^{(2)}&\boldsymbol{B}_{3}^{(3)}\begin{pmatrix}
					\widetilde{S}_{\boldsymbol{A}}^{(2)}\epsilon\\
					0
				\end{pmatrix}&0&j+k+l\\
				0&I&0&0&0&r\\
				0&0&I&0&0&s\\
				0&0&0&I&0&t\\
				0&0&0&0&I&h\\
			\end{block}
		\end{blockarray},
			\end{align*}
		and unitary dual quaternion matrices
			\begin{align*}
		\boldsymbol{U}^{(3)}&=I_p,\\
		 \boldsymbol{V}^{(3)}&=\begin{blockarray}{ccccc}
			q-r-s-t&r&s&t\\
			\begin{block}{(cccc)c}
				I&0&-\widetilde{\bf C}_{1}\epsilon&0&q-r-s-t\\
				0&I&0&0&r\\
				\widetilde{\bf C}_{1}^{*}\epsilon&0&I&0&s\\
				0&0&0&I&t\\
			\end{block}
		\end{blockarray},
		\end{align*}
	where 
	\begin{equation*}
		\widetilde{\bf {C}}_{1}\epsilon=\begin{pmatrix}
			{\Sigma}_{c}^{(1)}\epsilon&0\\
			0&0
		\end{pmatrix}\boldsymbol{B}_{2}^{(3)}({\sf S}_{\boldsymbol{B}}^{(2)})^{-1}{\sf S}_{\boldsymbol{A}}^{(2)}.
	\end{equation*}
	Then, a series of blocked elementary transformations yield
	\begin{align*}
		\boldsymbol{A}^{(4)}&=\boldsymbol{P}^{(3)}\boldsymbol{A}^{(3)}\boldsymbol{Q}^{(3)}
		=\begin{blockarray}{cccccc}
			j+k+l &r&s&t&h\\ 
			\begin{block}{(cccc|c)c}
				I&0&0&0&0&j+k+l  \\ 
				0&I&0&0&0&r\\
				0&0&{\sf S}_{\boldsymbol{A}}^{(2)}&0&0&s\\
				0&0&0&\widetilde{S}_{\boldsymbol{A}}^{(2)}\epsilon&0&t\\
				0&0&0&0&0&t_{1}+\hat{t}\\
			\end{block}
		\end{blockarray},  
			\end{align*}
			\begin{align*}		
		\boldsymbol{B}^{(4)}&=\boldsymbol{P}^{(3)}\boldsymbol{B}^{(3)}\boldsymbol{U}^{(3)}
		=\begin{blockarray}{cccc}
			p-s-t&s&t\\ 
			\begin{block}{(ccc)c}
				\boldsymbol{B}_{1}^{(3)}&0&0&j+k+l\\
				\begin{pmatrix}
					{\Sigma}_{b}^{(2)}\epsilon&0\\
					0&0
				\end{pmatrix}&0&0&r  \\ 
				0&{\sf S}_{\boldsymbol{B}}^{(2)}&0&s  \\
				0&0&I&t+\hat{t}\\
				0&0&0&t_{1}\\
			\end{block}
		\end{blockarray},\\
		\boldsymbol{C}^{(4)}&=\boldsymbol{V}^{(3)}\boldsymbol{C}^{(3)}\boldsymbol{Q}^{(3)}
		=\begin{blockarray}{cccccc}
			j+k+l &r&s&t&h\\ 
			\begin{block}{(cccc|c)c}
				\begin{pmatrix}
					{\Sigma}_{c}^{(1)}\epsilon&0\\
					0&0
				\end{pmatrix}&0&0&0&0&q-r-s-t  \\ 
				0&I&0&0&0&r\\
				0&0&I&0&0&s\\
				0&0&0&I&0&t\\
			\end{block}
		\end{blockarray}.
	\end{align*}

	Step 4. Perform the DQSVD of $\boldsymbol{B}_{1}^{(3)}$ that is the $(1,1)$-block submatrix of $\boldsymbol{B}^{(4)}$. There exist unitary dual quaternion matrices $\boldsymbol{U}_{3}\in \mathbb{DQ}^{(j+k+l)\times (j+k+l)}$ and $\boldsymbol{V}_{3}\in \mathbb{DQ}^{(p-\alpha-r)\times (p-\alpha-r)}$ such that
	\begin{align*}
		\boldsymbol{U}_{3}\boldsymbol{B}_{1}^{(3)}\boldsymbol{V}_{3}=
		\begin{blockarray}{cccc}
			\begin{block}{(ccc)c}
				{\sf \Sigma}_{\boldsymbol{B}_{1}^{(3)}}&0&0&j  \\ 
				0&\hat{\Sigma}_{\boldsymbol{B}_{1}^{(3)}}\epsilon&0&k\\
				0&0&0&l\\
			\end{block}
		\end{blockarray}.
	\end{align*}
	
	Let $\lambda=j+k+l+r$ and 
	\begin{equation}\label{eq:afbtgm} 
		\alpha_{\lambda+i}=\frac{\theta_{i}^{2}}{(1+\theta_{i}^{2})^{1/2}}, \quad
		\beta_{\lambda+i}=\delta_{i}, \quad
		\gamma_{\lambda+i}=\frac{\theta_{i}}{(1+\theta_{i}^{2})^{1/2}}.
	\end{equation}
	It is easy to verify that $\{\alpha_{\lambda+i}\}$, $\{\beta_{\lambda+i}\}$, $\{\gamma_{\lambda+i}\}$ satisfy \eqref{RSVDsq}-\eqref{RSVDEQ2}.
	
	Let
	\begin{subequations}\label{PQUV}
		\begin{align}
			\label{eq:Sabc}
			{\sf S}_{\boldsymbol{C}}&={\rm{diag}}(\gamma_{\lambda+i}), \quad {\sf S}_{\boldsymbol{A}}={\sf S}_{\boldsymbol{A}}^{(2)}{\sf S}_{\boldsymbol{C}}, \quad {\sf S}_{\boldsymbol{B}}={\sf S}_{\boldsymbol{B}}^{(2)}, \\
			\label{eq:P4}
			\boldsymbol{P}^{(4)} &={\rm{diag}}(({\sf \Sigma}_{\boldsymbol{B}_{1}^{(3)}})^{-1},(\hat{\Sigma}_{\boldsymbol{B}_{1}^{(3)}})^{-1},I_{m-j-k}){\rm{diag}}(\boldsymbol{U}_{3},I_{m-j-k-l}),  \\
			\label{eq:Q4}
			\boldsymbol{Q}^{(4)}&={\rm{diag}}(I_{j+k+l+r},{\sf S}_{\boldsymbol{C}},I_{n-j-k-l-r-s}){\rm{diag}}({\sf \Sigma}_{\boldsymbol{B}_{1}^{(3)}},\hat{\Sigma}_{\boldsymbol{B}_{1}^{(3)}},I_{n-j-k}) {\rm{diag}}(\boldsymbol{U}_{3}^{*},I_{n-j-k-l}), \\
			\label{eq:U4} \boldsymbol{U}^{(4)}&={\rm{diag}}(\boldsymbol{V}_{3},I_{\alpha+r}), \quad  \boldsymbol{V}^{(4)}=I_q.
		\end{align}
	\end{subequations}
	Obviously, the matrices in \eqref{eq:Sabc} are diagonal, those in \eqref{eq:P4}-\eqref{eq:Q4} are nonsingular, and those in \eqref{eq:U4} are unitary. Applying \eqref{eq:P4}-\eqref{eq:U4} to \eqref{updateeq} we obtain the results as stated in \eqref{RSVDresu1}-\eqref{RSVDEQ2}, and
		\begin{align}\label{subAX}
			\hat{\bf B}\epsilon &=\begin{pmatrix}
				\Sigma_{b}^{(2)}\epsilon&0\\
				0&0
			\end{pmatrix}\boldsymbol{V}_{3},\quad
			\hat{\bf C}_{1}\epsilon =\begin{pmatrix}
				{\Sigma}_{c}^{(1)}\epsilon&0\\
				0&0
			\end{pmatrix}{\rm{diag}}({\sf \Sigma}_{\boldsymbol{B}_{x}},\hat{\Sigma}_{\boldsymbol{B}_{x}},I_{l})\boldsymbol{U}_{3}^{*}.	
		\end{align}
	We then complete the proof.
\end{proof}

The dual quaternion matrices $\boldsymbol{P}\in {\bf \mathbb{DQ}}^{m \times m}$ and $\boldsymbol{Q}\in {\bf \mathbb{DQ}}^{n \times n}$ in \cref{DQRSVD111} are not ensured to be nonsingular. If DQGSVD2 are replaced with DQGSVD1 in the proof of \cref{DQRSVD111}, then we can obtain the other form of the DQRSVD where the dual quaternion matrices  $\boldsymbol{P}, \boldsymbol{Q}$ are nonsingular, as stated in the following theorem.

\begin{theorem}[DQRSVD2]\label{RSVD2}
	Let $\boldsymbol{A} \in {\bf \mathbb{DQ}}^{m \times n}$, $\boldsymbol{B} \in ^{m \times p}$, and $\boldsymbol{C} \in {\bf \mathbb{DQ}}^{q \times n}$. There exist nonsingular dual quaternion matrices $\boldsymbol{P}\in {\bf \mathbb{DQ}}^{m \times m}$, $\boldsymbol{Q}\in {\bf \mathbb{DQ}}^{n \times n}$, and unitary dual quaternion matrices $\boldsymbol{U}\in {\bf \mathbb{DQ}}^{p\times p}$, $\boldsymbol{V}\in {\bf \mathbb{DQ}}^{q\times q}$ such that
	\begin{subequations}\label{Thm42ABC}
		\begin{align}
			\boldsymbol{PAQ}&=
			\begin{blockarray}{ccc}
				n-a_{1} &a_{1}\\ 
				\begin{block}{(cc)c}
					{\sf \Sigma}_{\boldsymbol{A}}& 0  &m_{1}  \\
					0&\begin{pmatrix}
						\Sigma_{\boldsymbol{A}}\epsilon&0\\
						0&0
					\end{pmatrix}\\
				\end{block}
			\end{blockarray},    \  \
			\boldsymbol{PBU}=
			\begin{blockarray}{cc}
				p\\
				\begin{block}{(c)c}
					{\sf \Sigma}_{\boldsymbol{B}}&m-b_{2}-b_{3}  \\ 
					\widetilde{\bf B}\epsilon&b_{2}\\
					0_{\boldsymbol{B}}&b_{3}  \\
				\end{block}
			\end{blockarray},  \\
			\boldsymbol{VCQ}&=
			\begin{blockarray}{cccc}
				n-c_{1} &c_{1}&c_{2}\\ 
				\begin{block}{(ccc)c}
					{\sf \Sigma}_{\boldsymbol{C}}&\widetilde{\bf C}\epsilon&0_{\boldsymbol{C}}&q  \\ 
				\end{block}
			\end{blockarray},
		\end{align}
	\end{subequations}
	where
		\begin{align*}
			{\sf \Sigma}_{\boldsymbol{A}}&=\begin{blockarray}{cccccc}
				j+k+l &r&s&t&\hat{s}_{1}\\ 
				\begin{block}{(ccccc)c}
					I&0&0&0&0&j+k+l  \\ 
					0&I&0&0&0&r  \\
					0&0&{\sf S}_{\boldsymbol{A}}&0&0&s\\
					0&0&0&\widetilde{S}_{\boldsymbol{A}}\epsilon&0&t\\
					0&0&0&0&0&\hat{t}\\
				\end{block}
			\end{blockarray},  
			{\sf \Sigma}_{\boldsymbol{B}} =
			\begin{blockarray}{cccc}
				p-s-\hat{t} &s&\hat{t}\\ 
				\begin{block}{(ccc)c}
					\begin{pmatrix}
						I&0&0\\
						0&I\epsilon&0\\
						0&0&0
					\end{pmatrix}&0&0&j+k+l  \\ 
					\hat{\bf B}\epsilon&0&0&r\\
					0&{\sf S}_{\boldsymbol{B}}&0&s  \\
					0&0&I&t+\hat{t}\\
				\end{block}
			\end{blockarray}  
		\end{align*}
	    \begin{align*}
			{\sf \Sigma}_{\boldsymbol{C}}
			&=\begin{blockarray}{cccccc}
				j+k+l &r&s&\hat{t}&{s}_{1}\\ 
				\begin{block}{(ccccc)c}
					\hat{\bf C}_{1}\epsilon&0&0&0&0&q-r-s-t-\hat{s}_{1}  \\ 
					0&I&0&0&0&r\\
					0&0&I&0&0&s\\
					0&0&0&I&0&\hat{t}\\
					0&0&0&0&I&{s}_{1}\\
				\end{block}
			\end{blockarray},
		\end{align*}
	and the submatrices $\widetilde{{\bf B}}$, $\widetilde{{\bf C}}$, $\hat{\bf B}$, $\hat{\bf C}_1$ are general quaternion matrices, $\hat{\bf C}_1\epsilon$ is given by \eqref{subAX}, ${\widetilde S}_{\boldsymbol{A}}={\rm{diag}}(d_{i})$ is a real matrix,  ${\sf S}_{\boldsymbol{A}}={\rm{diag}}(\alpha_{i})$, ${\sf S}_{\boldsymbol{B}}={\rm{diag}}(\beta_{i})$, ${\sf S}_{\boldsymbol{C}}={\rm{diag}}(\gamma_{i})$ are dual matrices, and $\alpha_i, \beta_i, \gamma_i$ satisfy \eqref{RSVDsq}-\eqref{RSVDEQ2}.  
\end{theorem}
\begin{proof}
	Let
	$\boldsymbol{A}^{(4)}=\boldsymbol{A},\ \boldsymbol{B}^{(4)}=\boldsymbol{B},\  \boldsymbol{C}^{(4)}=\boldsymbol{C}.$
	Repeat the operation of proof in \cref{DQRSVD111} except for replacing DQGSVD2 with DQGSVD1, we can eventually find nonsingular dual quaternion matrices $\boldsymbol{P}^{(4)}\in {\bf \mathbb{DQ}}^{m \times m}$, $\boldsymbol{Q}^{(4)}\in {\bf \mathbb{DQ}}^{n \times n}$, and unitary dual quaternion matrices $\boldsymbol{U}^{(4)}\in {\bf \mathbb{DQ}}^{p\times p}$, $\boldsymbol{V}^{(4)}\in {\bf \mathbb{DQ}}^{q\times q}$ such that
		\begin{align*}
		\boldsymbol{A}^{(5)}&=\boldsymbol{P}^{(4)}\boldsymbol{A}^{(4)}\boldsymbol{Q}^{(4)}
			=\begin{blockarray}{cccccccc}
				j+k+l &r&s&t&c_1&c_2&c_3\\ 
				\begin{block}{(ccccc|c|c)c}
					I&0&0&0&0&\hat{{\bf X}}_{ac}^{1}\epsilon&0&j+k+l  \\ 
					0&I&0&0&0&\hat{{\bf X}}_{ac}^{2}\epsilon&0&r\\
					0&0&{\sf S}_{\boldsymbol{A}}^{(2)}&0&0&\hat{{\bf X}}_{ac}^{3}\epsilon&0&s\\
					0&0&0&\widetilde{S}_{\boldsymbol{A}}^{(2)}\epsilon&0&\hat{{\bf X}}_{ac}^{4}\epsilon&0&t\\
					0&0&0&0&0&\hat{{\bf X}}_{ac}^{5}\epsilon&0&\hat{t}\\
					0&\hat{{\bf X}}_{ab}^{1}\epsilon&\hat{{\bf X}}_{ab}^{2}\epsilon&\hat{{\bf X}}_{ab}^{3}\epsilon&0&\hat{{\bf X}}_{ac}^{6}\epsilon&0&\hat{r}_1\\
					0&0&0&0&0&0&0&\hat{r}_2\\
				\end{block}
			\end{blockarray},
					\end{align*}
			\begin{align*}
		\boldsymbol{B}^{(5)}&=\boldsymbol{P}^{(4)}\boldsymbol{B}^{(4)}\boldsymbol{U}^{(4)}
			=\begin{blockarray}{cccc}
				p-s-\hat{t} &s&t+\hat{t}\\ 
				\begin{block}{(ccc)c}
					\begin{pmatrix}
						I&0&0\\
						0&I\epsilon&0\\
						0&0&0
					\end{pmatrix}&0&0&j+k+l  \\ 
					\hat{\bf B}\epsilon&0&0&r\\
					0&{\sf S}_{\boldsymbol{B}}&0&s  \\
					0&0&I&t\\
					\hat{{\bf X}}_{ba}^{1}\epsilon&\hat{{\bf X}}_{ba}^{2}\epsilon&\hat{{\bf X}}_{ba}^{3}\epsilon&\hat{r}_{1}\\
					0&0&0&\hat{r}_{2}\\
				\end{block}
			\end{blockarray}, \\
		\boldsymbol{C}^{(5)}&=\boldsymbol{V}^{(4)}\boldsymbol{C}^{(4)}\boldsymbol{Q}^{(4)}
			=\begin{blockarray}{cccccccc}
				j+k+l &r&s&t&{c}_{1}&c_{2}&c_{3}\\ 
				\begin{block}{(ccccccc)c}
					\hat{\bf C}_{1}\epsilon&0&0&0&0&\hat{{\bf X}}_{ca}^{1}&0&q-r-s-t-\hat{s}_{1}  \\ 
					0&I&0&0&0&\hat{{\bf X}}_{ca}^{2}\epsilon&0&r\\
					0&0&I&0&0&\hat{{\bf X}}_{ca}^{3}\epsilon&0&s\\
					0&0&0&I&0&\hat{{\bf X}}_{ca}^{4}\epsilon&0&{t}\\
					0&0&0&0&I&\hat{{\bf X}}_{ca}^{5}\epsilon&0&c_{1}\\
				\end{block}
			\end{blockarray}.
		\end{align*}

	Considering the blocked structure of $\boldsymbol{A}^{(5)}$, it is easy to perform the row and column blocked elementary transformations on $\boldsymbol{A}^{(5)}$ to eliminate the block matrices $\hat{\bf X}_{ac}^{i}$ and $\hat{\bf X}_{ab}^{j}$, $i=1,2,3,4,5$ and $j=1,2,3$. Donate the row and column transform matrices as $\boldsymbol{P}_{4}$ and $\boldsymbol{Q}_{4}$, respectively. 
	
	Let
		$\boldsymbol{P}^{(5)}=\boldsymbol{P}_{4},\ \boldsymbol{Q}^{(5)}=\boldsymbol{Q}_{4},\ \boldsymbol{U}^{(5)}=I_{p},\  \boldsymbol{V}^{(5)}=I_{q}.$
	Then,
		\begin{align*}
		\boldsymbol{A}^{(6)}&=\boldsymbol{P}^{(5)}\boldsymbol{A}^{(5)}\boldsymbol{Q}^{(5)}
			=\begin{blockarray}{cccccccc}
				j+k+l &r&s&t&c_1&c_2&c_3\\ 
				\begin{block}{(ccccc|c|c)c}
					I&0&0&0&0&0&0&j+k+l  \\ 
					0&I&0&0&0&0&0&r\\
					0&0&{\sf S}_{\boldsymbol{A}}^{(2)}&0&0&0&0&s\\
					0&0&0&\widetilde{S}_{\boldsymbol{A}}^{(2)}\epsilon&0&0&0&t\\
					0&0&0&0&0&0&0&\hat{t}\\
					0&0&0&0&0&\widetilde{{\bf X}}_{ac}^{6}\epsilon&0&\hat{r}_1\\
					0&0&0&0&0&0&0&\hat{r}_2\\
				\end{block}
			\end{blockarray},\\
		\boldsymbol{B}^{(6)}&=\boldsymbol{P}^{(5)}\boldsymbol{B}^{(5)}\boldsymbol{U}^{(5)}
			=\begin{blockarray}{cccc}
				p-s-\hat{t} &s&t+\hat{t}\\ 
				\begin{block}{(ccc)c}
					\begin{pmatrix}
						I&0&0\\
						0&I\epsilon&0\\
						0&0&0
					\end{pmatrix}&0&0&j+k+l  \\ 
					\hat{\bf B}\epsilon&0&0&r\\
					0&{\sf S}_{\boldsymbol{B}}&0&s  \\
					0&0&I&t\\
					\widetilde{{\bf X}}_{ba}^{1}\epsilon&\widetilde{{\bf X}}_{ba}^{2}\epsilon&\widetilde{{\bf X}}_{ba}^{3}\epsilon&\hat{r}_{1}\\
					0&0&0&\hat{r}_{2}\\
				\end{block}
			\end{blockarray},
					\end{align*}
			\begin{align*}
		\boldsymbol{C}^{(6)}&=\boldsymbol{V}^{(5)}\boldsymbol{C}^{(5)}\boldsymbol{Q}^{(5)}
			=\begin{blockarray}{cccccccc}
				j+k+l &r&s&t&{c}_{1}&c_{2}&c_{3}\\ 
				\begin{block}{(ccccccc)c}
					\hat{\bf C}_{1}\epsilon&0&0&0&0&\widetilde{{\bf X}}_{ca}^{1}&0&q-r-s-t-\hat{s}_{1}  \\ 
					0&I&0&0&0&\widetilde{{\bf X}}_{ca}^{2}\epsilon&0&r\\
					0&0&I&0&0&\widetilde{{\bf X}}_{ca}^{3}\epsilon&0&s\\
					0&0&0&I&0&\widetilde{{\bf X}}_{ca}^{4}\epsilon&0&{t}\\
					0&0&0&0&I&\widetilde{{\bf X}}_{ca}^{5}\epsilon&0&c_{1}\\
				\end{block}
			\end{blockarray}.
		\end{align*}
	
	Performing the standard quaternion SVD for the quaternion matrix $\widetilde{{\bf X}}_{ac}^{6}$, there exist two unitary quaternion matrices ${\bf U}_{4}, {\bf V}_{4}$ such that
	\begin{equation*}
		{\bf U}_{4}^{*}\widetilde{{\bf X}}_{ac}^{6}{\bf V}_{4}=\begin{pmatrix}
			\Sigma_{\boldsymbol{A}}&0\\
			0&0
		\end{pmatrix}.
	\end{equation*}
	
	Let
$\boldsymbol{P}^{(6)}={\rm diag}(I_{j+k+l+r+s+t+\hat{t}},{\bf U}_{4}^{*},I_{\hat{r}_{2}}),\
 \boldsymbol{Q}^{(6)}={\rm diag}(I_{j+k+l+r+s+t+c_{1}}, {\bf V}_{4},I_{c_{2}}),$ $\boldsymbol{U}^{(6)}=I_{p},  \  \boldsymbol{V}^{(6)}=I_{q}.$
It is straightforward to verify that
		\begin{align*}
\boldsymbol{A}^{(7)}&=\boldsymbol{P}^{(6)}\boldsymbol{A}^{(6)}\boldsymbol{Q}^{(6)}, \
	\boldsymbol{B}^{(7)}=\boldsymbol{P}^{(6)}\boldsymbol{B}^{(6)}\boldsymbol{U}^{(6)},\\
	\boldsymbol{C}^{(7)}&=\boldsymbol{V}^{(6)}\boldsymbol{C}^{(6)}\boldsymbol{Q}^{(6)}
	\end{align*}
satisfy \eqref{Thm42ABC}, and $\alpha_i, \beta_i, \gamma_i$ satisfy  \eqref{RSVDsq}-\eqref{RSVDEQ2}.  
\end{proof}

\begin{remark}
	For the DQRSVD of dual quaternion matrices, the proof demonstrates that when all submatrices established in the DQSVD satisfy ${\rm rank}(\boldsymbol{M}) ={\rm Arank}(\boldsymbol{M})$ for each submatrix $\boldsymbol{M}$, then the matrices $\boldsymbol{P}$ and $\boldsymbol{Q}$ are guaranteed to be nonsingular. Moreover, this condition ensures that the decomposed matrices retain a well-formed structure, preserving the essential algebraic and geometric properties of the decomposition.
\end{remark}

\section{Product-Product SVD of a dual quaternion matrix triplet}\label{sec:PPSVD}

Concerning the numerical solution of the following matrix approximation problem $$\min\limits_{{\rm rank}(\boldsymbol{X})\leq r}\|\boldsymbol{A}(\boldsymbol{B} - \boldsymbol{X})\boldsymbol{C}\|_{\rm F},$$
where $\boldsymbol{A}, \boldsymbol{B}, \boldsymbol{C}$ are given general dual quaternion matrices with compatible dimensions, $r$ is a given nonnegative integer. Generally, the DQSVD of the product $\boldsymbol{ABC}$ can be used to solve this problem. However, as mentioned in \cite{Zha}, if we directly form the product $\boldsymbol{ABC}$ we may lose some of the information contained in the original matrices, leading to significant errors and a substantial discrepancy between the calculated values and the true values. In this section, we propose PPSVD of dual quaternion matrix triplets $(\boldsymbol{A},\boldsymbol{B},\boldsymbol{C})$ corresponding the DQSVD of the product $\boldsymbol{ABC}$. It only involves orthogonal transformations and treats the matrices $\boldsymbol{A},\boldsymbol{B}$ and $\boldsymbol{C}$ separately whenever possible.

First of all, we introduce the following results that can be directly obtained from \cref{lm:4.1}.

\begin{lemma}\label{PrePPSVD}
	Let $\boldsymbol{A} \in {\bf \mathbb{DQ}}^{m \times n}$ with ${\rm rank}(\boldsymbol{A})=r$, ${\rm Arank}(\boldsymbol{A})=\hat{r}$. Then, there exist unitary dual quaternion matrices $\boldsymbol{U}_{1}, \boldsymbol{U}_{2}\in {\bf \mathbb{DQ}}^{m \times m},$ $\boldsymbol{V}_{1}, \boldsymbol{V}_{2}\in {\bf \mathbb{DQ}}^{n \times n}$ and nonsingular dual quaternion matrices $\boldsymbol{P}\in {\bf \mathbb{DQ}}^{m \times m}$, $\boldsymbol{Q}\in {\bf \mathbb{DQ}}^{n \times n}$ such that
	\begin{subequations}
		\begin{align}\label{eq:PrePPSVDa}
			\boldsymbol{U}_{1}^{*}\boldsymbol{A}&=\begin{blockarray}{cc}
				\begin{block}{(c)c} 
				\boldsymbol{A}_{1}& \hat{r}\\
				{\bf A}_{2}\epsilon & r-\hat{r}  \\
				0 & m-r  \\
				\end{block}
			\end{blockarray},\ 
		\quad \boldsymbol{A}\boldsymbol{V}_1=\begin{blockarray}{ccc}
			\hat{r} & r-\hat{r} & n-r \\
			\begin{block}{(ccc)}
				\hat{\boldsymbol{A}}_{1}&\hat{\bf{A}}_{2}\epsilon&0 \\
		  \end{block}
			\end{blockarray},\\
			\label{eq:PrePPSVDb}
			\boldsymbol{U}_{2}^{*}\boldsymbol{A}\boldsymbol{Q}&=\begin{pmatrix}
				I_{\hat{r}}&&\\
				&I_{r-\hat{r}}\epsilon&\\
				&&&0
			\end{pmatrix}, \quad\boldsymbol{P}^{-1}\boldsymbol{A}\boldsymbol{V}_{2}=\begin{pmatrix}
				I_{\hat{r}}&&\\
				&I_{r-\hat{r}}\epsilon&\\
				&&&0
			\end{pmatrix}.
		\end{align}
	\end{subequations}
	where $\boldsymbol{A}_{1}\in {\bf \mathbb{DQ}}^{\hat{r} \times n}$ and ${\bf A}_{2}\in {\bf \mathbb{Q}}^{(r-\hat{r}) \times n}$ are of full row rank, $\hat{\boldsymbol{A}}_{1}\in {\bf \mathbb{DQ}}^{m\times \hat{r}}$ and $\hat{{\bf A}}_{2}\in {\bf \mathbb{Q}}^{m\times (r-\hat{r})}$ are of full column rank.
\end{lemma}


\begin{theorem}\label{ThmDQPPRSVD}
	Let $\boldsymbol{A} \in {\bf \mathbb{DQ}}^{m \times n}$, $\boldsymbol{B} \in {\bf \mathbb{DQ}}^{n\times p}$, $\boldsymbol{C} \in {\bf \mathbb{DQ}}^{p \times q}$. Then, there exist unitary matrices $\boldsymbol{U}\in {\bf \mathbb{DQ}}^{m \times m}$, $\boldsymbol{V}\in {\bf \mathbb{DQ}}^{q \times q}$, and nonsingular matrices $\boldsymbol{P}\in {\bf \mathbb{DQ}}^{n \times n}$, $\boldsymbol{Q}\in {\bf \mathbb{DQ}}^{p\times p}$ such that
	\begin{subequations}\label{PPSVDmatx}
		\begin{align}
			\boldsymbol{U}^{*}\boldsymbol{A}\boldsymbol{P}&=\begin{pmatrix}
				I_{i}&0\\
				0&{\bf A}_{22}\epsilon
			\end{pmatrix},\quad
			\boldsymbol{P}^{-1}\boldsymbol{B}\boldsymbol{Q}=\begin{pmatrix}
				I_{j}&0&0\\
				0&\hat{{\bf B}}_{12}\epsilon&\hat{{\bf B}}_{13}\epsilon\\
				0&I_{k}&{\bf B}_{23}\epsilon\\
				0&0&\hat{{\bf B}}_{23}\epsilon
			\end{pmatrix},  \\
			\boldsymbol{Q}^{-1}\boldsymbol{C}\boldsymbol{V}&=\begin{pmatrix}
					{\sf \Sigma}&0&0&0\\
					0&\hat{{\bf C}}_{12} \epsilon&\hat{{\bf C}}_{13} \epsilon&\hat{{\bf C}}_{14} \epsilon\\
					{\bf C}_{21} \epsilon&I_{s}&0&{\bf C}_{24} \epsilon\\
					0&0&0&\hat{{\bf C}}_{24} \epsilon\\
					{\bf C}_{31} \epsilon&0&I_{t}&{\bf C}_{34} \epsilon\\
					0&0&0&\hat{{\bf C}}_{34} \epsilon\\
				\end{pmatrix},
		\end{align}
		\end{subequations}
		where ${\sf \Sigma}={\rm diag}(\delta_{1},\ldots,\delta_{r})$ with $\delta_{1}\ge\cdots\ge\delta_{r}>0$ being the appreciable singular values of the matrix product $\boldsymbol{ABC}$, and $r={\rm Arank}(\boldsymbol{ABC})$.
\end{theorem}
\begin{proof}
	The proof is constructive, comprising a compression phase and a subsequent triangularization phase.
	
	First, from \eqref{eq:PrePPSVDa} in \cref{PrePPSVD} there exists $\widetilde{\boldsymbol{U}}_{2}\in {\bf \mathbb{DQ}}^{n \times n}$ such that
	\begin{equation}\label{eq:ModA}
		\boldsymbol{A}\widetilde{\boldsymbol{U}}_{2}=\begin{blockarray}{cc}
			i& n-i\\
			\begin{block}{(cc)}
				\boldsymbol{A}_{11}^{(1)},&\hat{\bf A}_{12}^{(1)}\epsilon\\
			\end{block}
		\end{blockarray},
	\end{equation}
	where $\boldsymbol{A}_{11}^{(1)}\in {\bf \mathbb{DQ}}^{m\times i}$ is of full column rank. Correspondingly, partition $\widetilde{\boldsymbol{U}}_{2}^{*}\boldsymbol{B}$ as follows
	\begin{equation*}
		\widetilde{\boldsymbol{U}}_{2}^{*}\boldsymbol{B}=\begin{blockarray}{cc}
			\begin{block}{(c)c}
				\boldsymbol{B}_{1}&i\\
				\boldsymbol{B}_{2}& n-i\\
			\end{block}
		\end{blockarray}.
	\end{equation*}
	Applying \eqref{eq:PrePPSVDa} once again, there exists $\hat{\boldsymbol{U}}_{3}\in {\bf \mathbb{DQ}}^{p \times p}$ that compresses the columns of $\boldsymbol{B}_{1}$ such that 
	\begin{equation*}
		\boldsymbol{B}_{1}\hat{\boldsymbol{U}}_{3}=\begin{blockarray}{ccc}
			j&&\\
			\begin{block}{(ccc)}
				{\boldsymbol{B}}_{11}^{(1)},&{{\bf B}}_{12}\epsilon,&{{\bf B}}_{13}\epsilon\\
			\end{block}
		\end{blockarray},
	\end{equation*}
	where $\boldsymbol{B}_{11}^{(1)}\in {\bf \mathbb{DQ}}^{i \times j}$ has full column rank. Then the resulting $\boldsymbol{B}$ can be partitioned as follows
	\begin{equation*}
	\widetilde{\boldsymbol{U}}_{2}^{*}\boldsymbol{B}\hat{\boldsymbol{U}}_{3}=\begin{pmatrix}
			{\boldsymbol{B}}_{11}^{(1)}&{{\bf B}}_{12}\epsilon&{{\bf B}}_{13}\epsilon\\
			\boldsymbol{B}_{21}&\boldsymbol{B}_{22}&\boldsymbol{B}_{23}\\
		\end{pmatrix}.
	\end{equation*}
	Further from \eqref{eq:PrePPSVDa} there exists $\bar{\boldsymbol{U}}_{2}\in {\bf \mathbb{DQ}}^{(p-j) \times (p-j)}$ that compresses the columns of $(\boldsymbol{B}_{22},\  \boldsymbol{B}_{23})$ such that
	\begin{equation*}
		(\boldsymbol{B}_{22},\  \boldsymbol{B}_{23})\bar{\boldsymbol{U}}_{2}=\begin{blockarray}{cc}
			k& p-j-k\\
			\begin{block}{(cc)}
				\boldsymbol{B}_{22}^{(1)}&{\bf B}_{23}^{(1)}\epsilon\\
			\end{block}
		\end{blockarray},
	\end{equation*}
	where $\boldsymbol{B}_{22}^{(1)}\in {\bf \mathbb{DQ}}^{(n-i) \times k}$ has full column rank. As a result,
	\begin{equation}\label{eq:ModB}
		\widetilde{\boldsymbol{U}}_{2}^{*}\boldsymbol{B}\hat{\boldsymbol{U}}_{3}{\rm diag}(I_j, \bar{\boldsymbol{U}}_{2})=\begin{blockarray}{cccc}
			j&k& p-j-k \\
			\begin{block}{(ccc)c}
				{\boldsymbol{B}}_{11}^{(1)}&{{\bf B}}_{12}^{(1)}\epsilon&{{\bf B}}_{13}^{(1)}\epsilon&i\\
				\boldsymbol{B}_{21}^{(1)}&\boldsymbol{B}_{22}^{(1)}&{\bf B}_{23}^{(1)}\epsilon & n-i\\
			\end{block}
		\end{blockarray}.
	\end{equation}
	
		Let $\widetilde{\boldsymbol{U}}_{3}=\hat{\boldsymbol{U}}_{3}{\rm diag}(I_j, \bar{\boldsymbol{U}}_{2})$.
	Then partition $\widetilde{\boldsymbol{U}}_{3}^{*}\boldsymbol{C}$ as follows
	\begin{equation*}
		\widetilde{\boldsymbol{U}}_{3}^{*}\boldsymbol{C}=\begin{blockarray}{cc}
			\begin{block}{(c)c}
				\boldsymbol{C}_{1}&j\\
				\boldsymbol{C}_{2}&k\\
				\boldsymbol{C}_{3} & p-j-k \\
			\end{block}
		\end{blockarray}.
	\end{equation*} 
 Repeating the compression procedure to the partitioned $\widetilde{\boldsymbol{U}}_{3}^{*}\boldsymbol{C}$ according to \eqref{eq:PrePPSVDa}, we can similarly find a unitary matrix $\widetilde{\boldsymbol{U}}_{4}\in {\bf \mathbb{DQ}}^{q \times q}$ such that
	\begin{equation*}
		\widetilde{\boldsymbol{U}}_{3}^{*}\boldsymbol{C}\widetilde{\boldsymbol{U}}_{4}=\begin{blockarray}{ccccc}
			\begin{block}{(cccc)c}
				\boldsymbol{C}_{11}^{(1)}&{\bf C}_{12}^{(1)}\epsilon&{\bf C}_{13}^{(1)}\epsilon&{\bf C}_{14}^{(1)}\epsilon&j\\
				\boldsymbol{C}_{21}^{(1)}&\boldsymbol{C}_{22}^{(1)}&{\bf C}_{23}^{(1)}\epsilon&{\bf C}_{24}^{(1)}\epsilon&k\\
				\boldsymbol{C}_{31}^{(1)}&\boldsymbol{C}_{32}^{(1)}&\boldsymbol{C}_{33}^{(1)}&{\bf C}_{34}^{(1)}\epsilon& p-j-k\\
			\end{block}
		\end{blockarray},
	\end{equation*}
	where $\boldsymbol{C}_{11}^{(1)}$, $\boldsymbol{C}_{22}^{(1)}$, $\boldsymbol{C}_{33}^{(1)}$ are dual quaternion matrices with full column rank. 
	
	After completing the forward procedure, we proceed to the backward triangularization process. Triangularize the matrices $\boldsymbol{C}_{11}^{(1)}$, $\boldsymbol{C}_{22}^{(1)}$, $\boldsymbol{C}_{33}^{(1)}$ using the DQQR decomposition \cite{LingMW} such that
	\begin{equation*}
		\boldsymbol{C}_{11}^{(1)}=\boldsymbol{Q}_{c}^{(1)}\begin{pmatrix}
			\boldsymbol{C}_{11}^{(2)}\\
			0
		\end{pmatrix},\quad \boldsymbol{C}_{22}^{(1)}=\boldsymbol{Q}_{c}^{(2)}\begin{pmatrix}
			\boldsymbol{C}_{22}^{(2)}\\
			0
		\end{pmatrix},\quad \boldsymbol{C}_{33}^{(1)}=\boldsymbol{Q}_{c}^{(3)}\begin{pmatrix}
			\boldsymbol{C}_{33}^{(2)}\\
			0
		\end{pmatrix},
	\end{equation*} 
	where $\boldsymbol{C}_{11}^{(2)}$, $\boldsymbol{C}_{22}^{(2)}$, $\boldsymbol{C}_{33}^{(2)}$ are nonsingular upper triangular dual quaternion matrices. Modifying the corresponding columns of the resulted $\boldsymbol{B}$ in \eqref{eq:ModB} with $\boldsymbol{Q}_{c}^{(1)}$, $\boldsymbol{Q}_{c}^{(2)}$, $\boldsymbol{Q}_{c}^{(3)}$ we obtain
	\begin{equation*}
		\begin{pmatrix}
			{\boldsymbol{B}}_{11}^{(1)}&{{\bf B}}_{12}^{(1)}\epsilon&{{\bf B}}_{13}^{(1)}\epsilon\\
			\boldsymbol{B}_{21}^{(1)}&\boldsymbol{B}_{22}^{(1)}&{\bf B}_{23}^{(1)}\epsilon
		\end{pmatrix}\begin{pmatrix}
			\boldsymbol{Q}_{c}^{(1)}&&\\
			&\boldsymbol{Q}_{c}^{(2)}&\\
			&&\boldsymbol{Q}_{c}^{(3)}
		\end{pmatrix}=\begin{pmatrix}
			\boldsymbol{B}_{11}^{(2)}&{\bf B}_{12}^{(2)}\epsilon&{\bf B}_{13}^{(2)}\epsilon\\
			\boldsymbol{B}_{21}^{(2)}&\boldsymbol{B}_{22}^{(2)}&{\bf B}_{23}^{(2)}\epsilon
		\end{pmatrix}.
	\end{equation*}
	Similarly triangularize the matrices $\boldsymbol{B}_{11}^{(2)}$ and $\boldsymbol{B}_{22}^{2}$ using DQQR decompositionc \cite{LingMW} to obtain
	\begin{equation*}
		\boldsymbol{B}_{11}^{(2)}=\boldsymbol{Q}_{b}^{(1)}\begin{pmatrix}
			\boldsymbol{B}_{11}^{(3)}\\
			0
		\end{pmatrix},\quad\boldsymbol{B}_{22}^{(2)}=\boldsymbol{Q}_{b}^{(2)}\begin{pmatrix}
			\boldsymbol{B}_{22}^{(3)}\\
			0
		\end{pmatrix},
	\end{equation*}
	where $\boldsymbol{B}_{11}^{(3)}$ and $\boldsymbol{B}_{22}^{(3)}$ are nonsingular upper triangular dual quaternion matrices. Modify the corresponding columns of the resulted $\boldsymbol{A}$ in \eqref{eq:ModA} with $\boldsymbol{Q}_{b}^{(1)}$ and $\boldsymbol{Q}_{b}^{(2)}$ to obtain 
	\begin{equation*}
		(\boldsymbol{A}_{11}^{(1)},\quad{\bf A}_{12}^{(1)}\epsilon)\begin{pmatrix}
			\boldsymbol{Q}_{b}^{(1)}&\\
			&\boldsymbol{Q}_{b}^{(2)}
		\end{pmatrix}=(\boldsymbol{A}_{11}^{(2)},\quad{\bf A}_{12}^{(2)}\epsilon).
	\end{equation*}
	Triangularize $\boldsymbol{A}_{11}^{(2)}$ using DQQR decomposition so that
	\begin{equation*}
		\boldsymbol{A}_{11}^{(2)}=\boldsymbol{Q}_{a}\begin{pmatrix}
			\boldsymbol{A}_{11}^{(3)}\\
			0
		\end{pmatrix}, 
	\end{equation*}
	where $\boldsymbol{A}_{11}^{(3)}$ is a nonsingular upper triangular dual quaternion matrix.
	
	Donate 
		\begin{align}
			\widetilde{\boldsymbol{C}}&=\begin{pmatrix}
				\boldsymbol{Q}_{c}^{(1)}&&\\
				&\boldsymbol{Q}_{c}^{(2)}&\\
				&&\boldsymbol{Q}_{c}^{(3)}\\
			\end{pmatrix}^{*}\widetilde{\boldsymbol{U}}_{3}^{*}\boldsymbol{C}\widetilde{\boldsymbol{U}}_{4}\nonumber\\
			&\label{eq:Cb}=\begin{blockarray}{ccccc}
				\begin{block}{(cccc)c}
					\begin{pmatrix}
						\boldsymbol{C}_{11}^{(2)}\\
						0
					\end{pmatrix}&{\bf C}_{12}^{(2)}\epsilon&{\bf C}_{13}^{(2)}\epsilon&{\bf C}_{14}^{(2)}\epsilon\\
					\boldsymbol{C}_{21}^{(2)}&\begin{pmatrix}
						\boldsymbol{C}_{22}^{(2)}\\
						0
					\end{pmatrix}&{\bf C}_{23}^{(2)}\epsilon&{\bf C}_{24}^{(2)}\epsilon\\
					\boldsymbol{C}_{31}^{(2)}&\boldsymbol{C}_{32}^{(2)}&\begin{pmatrix}
						\boldsymbol{C}_{33}^{(2)}\\
						0
					\end{pmatrix}&{\bf C}_{34}^{(2)}\epsilon\\
				\end{block}
			\end{blockarray},\\
			\widetilde{\boldsymbol{B}}&=\begin{pmatrix}
				\boldsymbol{Q}_{b}^{(1)}&\\
				&\boldsymbol{Q}_{b}^{(2)}
			\end{pmatrix}^{*}\widetilde{\boldsymbol{U}}_{2}^{*}\boldsymbol{B}\widetilde{\boldsymbol{U}}_{3}\begin{pmatrix}
				\boldsymbol{Q}_{c}^{(1)}&&\\
				&\boldsymbol{Q}_{c}^{(2)}&\\
				&&\boldsymbol{Q}_{c}^{(3)}
			\end{pmatrix}\nonumber\\
			&\label{eq:Bd}=\begin{pmatrix}
				\begin{pmatrix}
					\boldsymbol{B}_{11}^{(3)}\\
					0
				\end{pmatrix}&{\bf B}_{12}^{(3)}\epsilon&{\bf B}_{13}^{(3)}\epsilon\\
				\boldsymbol{B}_{21}^{(3)}&\begin{pmatrix}
					\boldsymbol{B}_{22}^{(3)}\\
					0
				\end{pmatrix}&{\bf B}_{23}^{(3)}\epsilon
			\end{pmatrix},
					\end{align}
			\begin{align}
		\label{eq:Ae}
			\widetilde{\boldsymbol{A}}&=\boldsymbol{Q}_{a}^{*}(\boldsymbol{A}_{11}^{(2)},\ {\bf A}_{12}^{(2)}\epsilon)=\begin{pmatrix}
				\boldsymbol{A}_{11}^{(3)}&{\bf A}_{12}^{(3)}\epsilon\\
				0&{\bf A}_{22}^{(3)}\epsilon
			\end{pmatrix}.
		\end{align}
	
	Since $\boldsymbol{C}_{11}^{(2)}$, $\boldsymbol{C}_{22}^{(2)}$, $\boldsymbol{C}_{33}^{(2)}$ are nonsingular upper triangular dual quaternion matrices in \eqref{eq:Cb}, they can be used as pivoting matrices to eliminate the block lower triangular part of $\widetilde{\boldsymbol{C}}$. Therefore, it is easy to find a blocked elementary dual quaternion matrix $\boldsymbol{G}_{1}$ such that
	\begin{equation*}
		\boldsymbol{G}_{1}^{-1}\widetilde{\boldsymbol{C}}=\begin{blockarray}{ccccc}
			\begin{block}{(cccc)c}
				\boldsymbol{C}_{11}^{(3)}&0&0&{\bf C}_{14}^{(3)}\epsilon\\
				0&0&0&\hat{{\bf C}}_{14}^{(3)}\epsilon\\
				0&\boldsymbol{C}_{22}^{(3)}&0&{\bf C}_{24}^{(3)}\epsilon\\
				0&0&0&\hat{{\bf C}}_{24}^{(3)}\epsilon\\
				0&0&\boldsymbol{C}_{33}^{(3)}&{\bf C}_{34}^{(3)}\epsilon\\
				0&0&0&\hat{{\bf C}}_{34}^{(3)}\epsilon\\
			\end{block}
		\end{blockarray}.
	\end{equation*}
	Applying $\boldsymbol{G}_{1}$ to the right-hand side of $\widetilde{\boldsymbol{B}}$ in \eqref{eq:Bd} such that
	\begin{equation*}
		\widetilde{\boldsymbol{B}}\boldsymbol{G}_{1}=\begin{pmatrix}
			\begin{pmatrix}
				\boldsymbol{B}_{11}^{(4)}\\
				\hat{{\bf B}}_{11}^{(4)}\epsilon
			\end{pmatrix}&{\bf B}_{12}^{(4)}\epsilon&{\bf B}_{13}^{(4)}\epsilon\\
			\boldsymbol{B}_{21}^{(4)}&\begin{pmatrix}
				\boldsymbol{B}_{22}^{(4)}\\
				\hat{\bf B}_{22}^{(4)}\epsilon
			\end{pmatrix}&{\bf B}_{23}^{(4)}\epsilon
		\end{pmatrix}.
	\end{equation*}
	Similarly, there exists a blocked elementary dual quaternion matrix $\boldsymbol{G}_{2}$ such that
	\begin{equation*}
		\boldsymbol{G}_{2}^{-1}\widetilde{\boldsymbol{B}}\boldsymbol{G}_{1}=\begin{pmatrix}
			\boldsymbol{B}_{11}^{(5)}&0&{\bf B}_{13}^{(5)}\epsilon\\
			0&0&\hat{{\bf B}}_{13}^{(5)}\epsilon\\
			0&\boldsymbol{B}_{22}^{(5)}&{\bf B}_{23}^{(5)}\epsilon\\
			0&0&{\bf B}_{23}^{(5)}\epsilon
		\end{pmatrix}.
	\end{equation*}
	
	Multiplying $\boldsymbol{G}_{2}$ on the right-hand side of $\widetilde{\boldsymbol{A}}$ in \eqref{eq:Ae} gives
	\begin{equation*}
		\widetilde{\boldsymbol{A}}\boldsymbol{G}_{2}=\begin{pmatrix}
			\boldsymbol{A}_{11}^{(4)}&{\bf A}_{12}^{(4)}\epsilon\\
			{\bf A}_{12}^{(4)}\epsilon&{\bf A}_{22}^{(4)}\epsilon
		\end{pmatrix}.
	\end{equation*}
	With the DQQR decomposition of the first block column of $\widetilde{\boldsymbol{A}}\boldsymbol{G}_{2}$, there exists a unitary dual quaternion matrix $\widetilde{\boldsymbol{U}}_{1}$ such that
	\begin{equation*}
		\widetilde{\boldsymbol{U}}_{1}\widetilde{\boldsymbol{A}}\boldsymbol{G}_{2}=\begin{pmatrix}
			\boldsymbol{A}_{11}^{(5)}&{\bf A}_{12}^{(5)}\epsilon\\
			0&{\bf A}_{22}^{(5)}\epsilon
		\end{pmatrix}.
	\end{equation*}
	
	According to \eqref{eq:PrePPSVDb} in \cref{PrePPSVD} there exists a nonsingular dual quaternion matrix $\boldsymbol{P}_{1}$ and a unitary dual quaternion matrix $\boldsymbol{U}_1$ such that $\boldsymbol{U}_{1}^{*}\boldsymbol{A}_{11}^{(5)}\boldsymbol{P}_{1}=I_{i}$. Similarly, we  obtain
	\begin{equation*}
		\boldsymbol{U}_{2}^{*}\left(\boldsymbol{P}_{1}^{-1}\begin{pmatrix}
			\boldsymbol{B}_{11}^{(5)}\\
			0
		\end{pmatrix}\right)\boldsymbol{Q}_{1}=\begin{pmatrix}
			I_{j}\\
			0
		\end{pmatrix},\quad\boldsymbol{P}_{2}^{-1}\begin{pmatrix}
			\boldsymbol{B}_{22}\\
			0
		\end{pmatrix}\boldsymbol{Q}_{2}=\begin{pmatrix}
			I_{k}\\
			0
		\end{pmatrix},
	\end{equation*} 
	and
	\begin{subequations}
		\begin{align*}	
			\boldsymbol{U}_{3}^{*}\left(\boldsymbol{Q}_{1}^{-1}\begin{pmatrix}
				\boldsymbol{C}_{11}^{(3)}\\
				0
			\end{pmatrix}\right)\boldsymbol{V}_{1}&=\begin{pmatrix}
				{\sf \Sigma}\\
				0
			\end{pmatrix},\ \
			\boldsymbol{Q}_{3}^{-1}\left(\boldsymbol{Q}_{2}^{-1}\begin{pmatrix}
				\boldsymbol{C}_{22}^{(3)}\\
				0
			\end{pmatrix}\right)\boldsymbol{V}_{2}=\begin{pmatrix}
				I_{s}\\
				0
			\end{pmatrix},\\
			\boldsymbol{Q}_{4}^{-1}\begin{pmatrix}
				\boldsymbol{C}_{33}^{(3)}\\
				0
			\end{pmatrix}\boldsymbol{V}_{3}&=\begin{pmatrix}
				I_{t}\\
				0
			\end{pmatrix}.
		\end{align*}
	\end{subequations}
	
	As a result, we have
		\begin{align}
			\hat{\boldsymbol{A}}&=\begin{pmatrix}
				\boldsymbol{U}_{3}^{*}&\\
				&I_{m-j}
			\end{pmatrix}\begin{pmatrix}
				\boldsymbol{U}_{2}^{*}&\\
				&I_{m-i}
			\end{pmatrix}\begin{pmatrix}
				\boldsymbol{U}_{1}^{*}&\\
				&I_{m-i}
			\end{pmatrix}\widetilde{\boldsymbol{U}}_{1}\widetilde{\boldsymbol{A}}\boldsymbol{G}_{2}\nonumber\\
			& \quad \cdot  \begin{pmatrix}
				\boldsymbol{P}_{1}\boldsymbol{U}_{2}&\\
				&\boldsymbol{P}_{2}
			\end{pmatrix} \begin{pmatrix}
				\boldsymbol{U}_{3}&&&\\
				&I_{i-j}&&\\
				&&\boldsymbol{Q}_{3}&\\
				&&&I_{n-i-k}
			\end{pmatrix}\nonumber \\
			&\label{eq:hatAc}=\begin{pmatrix}
				I_{i}&{\bf A}_{12}^{(6)}\epsilon\\
				0&{\bf A}_{22}^{(6)}\epsilon
			\end{pmatrix},
			\end{align}
			\begin{align}
			\hat{\boldsymbol{B}}&=\begin{pmatrix}
				\boldsymbol{U}_{3}^{*}&&&\\
				&I_{i-j}&&\\
				&&\boldsymbol{Q}_{3}^{-1}&\\
				&&&I_{n-i-k}
			\end{pmatrix}\begin{pmatrix}
				\boldsymbol{U}_{2}^{*}\boldsymbol{P}_{1}^{-1}&\\
				&\boldsymbol{P}_{2}^{-1}
			\end{pmatrix}\boldsymbol{G}_{2}^{-1}\widetilde{\boldsymbol{B}}\boldsymbol{G}_{1}\nonumber\\
			&\quad \cdot \begin{pmatrix}
				\boldsymbol{Q}_{1}\boldsymbol{U}_{3}&&\\
				&\boldsymbol{Q}_{2}\boldsymbol{Q}_{3}&\\
				&&\boldsymbol{Q}_{4}
			\end{pmatrix}\nonumber\\
			&\label{eq:hatBf}=\begin{pmatrix}
				I_{j}&0&{\bf B}_{13}^{(6)}\epsilon\\
				0&0&\hat{{\bf B}}_{13}^{(6)}\epsilon\\
				0&I_{k}&{\bf B}_{23}^{(6)}\epsilon\\
				0&0&\hat{{\bf B}}_{23}^{(6)}\epsilon
			\end{pmatrix},\\
			\hat{\boldsymbol{C}}&=\begin{pmatrix}
				\boldsymbol{U}_{3}^{*}\boldsymbol{Q}_{1}^{-1}&&\\
				&\boldsymbol{Q}_{3}^{-1}\boldsymbol{Q}_{2}^{-1}&\\
				&&\boldsymbol{Q}_{4}^{-1}
			\end{pmatrix}\boldsymbol{G}_{1}^{-1}\widetilde{\boldsymbol{C}}\begin{pmatrix}
				\boldsymbol{V}_{1}&&\\
				&\boldsymbol{V}_{2}&\\
				&&\boldsymbol{V}_{3} \\
				&&& I_{q-l-s-t}
			\end{pmatrix}\nonumber\\
			&\label{eq:hatCh}=\begin{blockarray}{ccccc}
				\begin{block}{(cccc)c}
					{\sf \Sigma}&0&0&{\bf C}_{14}^{(3)}\epsilon&l\\
					0&0&0&\hat{{\bf C}}_{14}^{(3)}\epsilon&j-l\\
					0&I_{s}&0&{\bf C}_{24}^{(3)}\epsilon&s\\
					0&0&0&\hat{{\bf C}}_{24}^{(3)}\epsilon&k-s\\
					0&0&I_{t}&{\bf C}_{34}^{(3)}\epsilon&t\\
					0&0&0&\hat{{\bf C}}_{34}^{(3)}\epsilon&p-j-k-t\\
				\end{block}
			\end{blockarray}.
		\end{align}

	Applying the $(1, 1)$-block matrix to eliminate the $(1, 2)$-block of $\hat{\boldsymbol{A}}$ in \eqref{eq:hatAc}, it is easy to find a blocked elementary dual quaternion matrix $\boldsymbol{G}_{3}$ such that
	\begin{equation*}
		\hat{\boldsymbol{A}}\boldsymbol{G}_{3}=\begin{pmatrix}
			I_{i}&0\\
			0&{\bf A}_{22}^{(7)}\epsilon
		\end{pmatrix}.
	\end{equation*}
	Multiplying $\boldsymbol{G}_{3}^{-1}$ on the left-hand side of $\hat{\boldsymbol{B}}$ in \eqref{eq:hatBf} yields
	\begin{equation}\label{eq:GB}
		\boldsymbol{G}_{3}^{-1}\hat{\boldsymbol{B}}=\begin{pmatrix}
			I_{j}&{\bf B}_{12}^{(7)}\epsilon&{\bf B}_{13}^{(7)}\epsilon\\
			0&\hat{{\bf B}}_{12}^{(7)}\epsilon&\hat{{\bf B}}_{13}^{(7)}\epsilon\\
			0&I_{k}&{\bf B}_{23}^{(7)}\epsilon\\
			0&0&\hat{{\bf B}}_{23}^{(7)}\epsilon
		\end{pmatrix},
	\end{equation}
	further applying a blocked elementary matrix $\boldsymbol{G}_{4}$ to eliminate the nondiagonal blocks in the first blocked row of \eqref{eq:GB} such that
	\begin{equation*}
		\boldsymbol{G}_{3}^{-1}\hat{\boldsymbol{B}}\boldsymbol{G}_{4}=\begin{pmatrix}
			I_{j}&0&0\\
			0&\hat{{\bf B}}_{12}^{(7)}\epsilon&\hat{{\bf B}}_{13}^{(7)}\epsilon\\
			0&I_{k}&{\bf B}_{23}^{(7)}\epsilon\\
			0&0&\hat{{\bf B}}_{23}^{(7)}\epsilon
		\end{pmatrix}.
	\end{equation*}
	
	Multiplying $\boldsymbol{G}_{4}^{-1}$ on the left-hand side of $\hat{\boldsymbol{C}}$ in \eqref{eq:hatCh} gives
	\begin{equation*}
		\boldsymbol{G}_{4}^{-1}\hat{\boldsymbol{C}}=\begin{blockarray}{ccccc}
			\begin{block}{(cccc)c}
				{\sf \Sigma}&{\bf C}_{12}^{(4)}\epsilon&{\bf C}_{13}^{(4)}\epsilon&{\bf C}_{14}^{(4)}\epsilon\\
				0&\hat{{\bf C}}_{12}^{(4)}\epsilon&\hat{{\bf C}}_{13}^{(4)}\epsilon&\hat{{\bf C}}_{14}^{(4)}\epsilon\\
				0&I_{s}&0&{\bf C}_{24}^{(4)}\epsilon\\
				0&0&0&\hat{{\bf C}}_{24}^{(4)}\epsilon\\
				0&0&I_{t}&{\bf C}_{34}^{(4)}\epsilon\\
				0&0&0&\hat{{\bf C}}_{34}^{(4)}\epsilon\\
			\end{block}
		\end{blockarray},
	\end{equation*}
	further multiplying a unitary dual quaternion matrix 
	 \begin{equation*}
	 	\boldsymbol{V}_{4}=\begin{pmatrix}
	 		I_l &  -{\sf \Sigma}^{-1}{\bf C}_{12}^{(4)}\epsilon & -{\sf \Sigma}^{-1}{\bf C}_{13}^{(4)}\epsilon & -{\sf \Sigma}^{-1}{\bf C}_{14}^{(4)}\epsilon \\
	 		({\sf \Sigma}^{-1}{\bf C}_{12}^{(4)})^*\epsilon  &  I_s  &  0  &  0 \\
	 		({\sf \Sigma}^{-1}{\bf C}_{13}^{(4)})^*\epsilon  &  0  &  I_t  &  0  \\
	 		({\sf \Sigma}^{-1}{\bf C}_{14}^{(4)})^*\epsilon & 0 & 0 & I_{q-l-s-t}
	 	\end{pmatrix}
	 \end{equation*}
	 with the same form as in \cref{eq:DM1} on the right-hand side such that
	\begin{equation*}
		\boldsymbol{G}_{4}^{-1}\hat{\boldsymbol{C}}\boldsymbol{V}_{4}=\begin{blockarray}{ccccc}
			\begin{block}{(cccc)c}
				{\sf \Sigma}&0&0&0\\
				0&\hat{{\bf C}}_{12}^{(5)}\epsilon&\hat{{\bf C}}_{13}^{(5)}\epsilon&\hat{{\bf C}}_{14}^{(5)}\epsilon\\
				{\bf C}_{21}^{(5)}\epsilon&I_{s}&0&{\bf C}_{24}^{(5)}\epsilon\\
				0&0&0&\hat{{\bf C}}_{24}^{(5)}\epsilon\\
				{\bf C}_{31}^{(5)}\epsilon&0&I_{t}&{\bf C}_{34}^{(5)}\epsilon\\
				0&0&0&\hat{{\bf C}}_{34}^{(5)}\epsilon\\
			\end{block}
		\end{blockarray}.
	\end{equation*}

	Accumulating the transformation matrices gives
	\begin{subequations}
		\begin{align*}
			\boldsymbol{U}&=\boldsymbol{Q}_{a}\widetilde{\boldsymbol{U}}_{1}{\rm diag}({\boldsymbol{U}}_{1},I_{m-i}){\rm diag}(\boldsymbol{U}_{2},I_{m-i}){\rm diag}(\boldsymbol{U}_{3},I_{m-j}),\\
			\boldsymbol{V}&=\widetilde{\boldsymbol{U}}_{4}{\rm diag}(\boldsymbol{V}_{1},I_{q-l}){\rm diag}(I_{l},\boldsymbol{V}_{2},I_{q-l-s}){\rm diag}(I_{l+s},\boldsymbol{V}_{3},I_{q-l-s-t})\boldsymbol{V}_{4},\\
			\boldsymbol{P}&=\widetilde{\boldsymbol{U}}_{2}{\rm diag}(\boldsymbol{Q}_{b}^{(1)},I_{n-i}){\rm diag}(I_{i},\boldsymbol{Q}_{b}^{(2)})\boldsymbol{G}_{2}{\rm diag}(\boldsymbol{P}_{1}\boldsymbol{U}_{2},\boldsymbol{P}_{2})\\
			&\quad \cdot {\rm diag}(\boldsymbol{U}_{3},I_{i-j},\boldsymbol{Q}_{3},I_{n-i-k})\boldsymbol{G}_{3},\\
			\boldsymbol{Q}&=\widetilde{\boldsymbol{U}}_{3}{\rm diag}(\boldsymbol{Q}_{c}^{(1)},\boldsymbol{Q}_{c}^{(2)},\boldsymbol{Q}_{c}^{(3)})\boldsymbol{G}_{1}{\rm diag}(\boldsymbol{Q}_{1}\boldsymbol{U}_{3},\boldsymbol{Q}_{2}\boldsymbol{Q}_{3},\boldsymbol{Q}_{4})\boldsymbol{G}_{4}.
		\end{align*}
	\end{subequations}
	Then we can obtain 
	\begin{equation*}
		\boldsymbol{U}^{*}(\boldsymbol{ABC})\boldsymbol{V}=\begin{pmatrix}
			{\sf \Sigma}&\\
			&{\bf T}\epsilon
		\end{pmatrix},
	\end{equation*}
	where ${\sf \Sigma} \in \mathbb{D}^{l\times l}$ is a diagonal matrix with positive elements, ${\bf T}$ is an $(m-l)\times (q-l)$ quaternion matrix.	
	
	Finally, for the standard quaternion SVD of ${\bf T}$, there exist unitary quaternion matrices $\hat{\bf U}\in \mathbb{Q}^{(m-l)\times (m-l)}$ and $\hat{\bf V}\in \mathbb{Q}^{(q-l)\times (q-l)}$, such that
	\begin{equation}\label{svdofT}
	\hat{\bf{U}}^{*}{\bf T}\hat{\bf{V}}=\begin{pmatrix}
		\hat{\Sigma}&0\\
		0&0
	\end{pmatrix}.
	\end{equation}
	Let $\widetilde{\boldsymbol{U}}=\boldsymbol{U}\begin{pmatrix}
		I_l&0\\
		0&\hat{\bf{U}}
	\end{pmatrix}$ and $\widetilde{\boldsymbol{V}}=\boldsymbol{V}\begin{pmatrix}
	I_l&0\\
	0&\hat{\bf{V}}
	\end{pmatrix}$. Then we obtain DQSVD of the product $\boldsymbol{ABC}$ to be
	 \begin{equation}\label{DQSVDofabc}
	 	\widetilde{\boldsymbol{U}}^\ast(\boldsymbol{ABC})\widetilde{\boldsymbol{V}}=\begin{pmatrix}
	 		{\sf \Sigma}&&\\
	 		&\hat{\Sigma}\epsilon&\\
	 		&&&0
	 	\end{pmatrix}.
	 \end{equation}
	
\end{proof}
\begin{remark}
	Unlike the PPSVD of a real matrix triplet \cite{Zha}, the results of the PPSVD of a dual quaternion matrix triplet in \eqref{PPSVDmatx} can not straightforwardly reveal the singular values of the product matrix $\boldsymbol{ABC}$; they only display the appreciable singular values. With one more step about the standard quaternion SVD of ${\bf T}$ in \eqref{svdofT}, one can find all the singular values of $\boldsymbol{ABC}$, as can be seen from \eqref{DQSVDofabc}. Moreover, if $\boldsymbol{A}$, $\boldsymbol{B}$ and $\boldsymbol{C}$ satisfy ${\rm rank}(\boldsymbol{A})={\rm Arank}(\boldsymbol{A})$, ${\rm rank}(\boldsymbol{B})={\rm Arank}(\boldsymbol{B})$ and ${\rm rank}(\boldsymbol{C})={\rm Arank}(\boldsymbol{C})$, then the purely dual submatrices will vanish in the aforementioned decompositions in \eqref{PPSVDmatx}.
\end{remark}


\section{Two illustrative examples}\label{sec:exm}
In this section, we present two artificial examples to illustrate the restricted SVD and the PPSVD of dual quaternion matrix triplets.

\begin{example}
Find the restricted SVD of the following dual quaternion triplet $(\boldsymbol{A},\boldsymbol{B},\boldsymbol{C})$, where
\begin{align*}
		\boldsymbol{A}&=\begin{pmatrix}
			1&\frac{\sqrt{2}}{2}+(\frac{\sqrt{2}}{2}{\bf k}+1)\epsilon&0\\
			{\bf k}\epsilon&\frac{\sqrt{2}}{2}+(\frac{\sqrt{2}}{2}{\bf k}+1)\epsilon&0\\
			0&0&\epsilon
		\end{pmatrix},\
		\boldsymbol{B}=\begin{pmatrix}
			1+({\bf i}+{\bf k})\epsilon&1+2{\bf j}+2{\bf k}\epsilon\\
			1+({\bf k}+1)\epsilon&1+(\sqrt{2}-1-2{\bf i}+{\bf k})\epsilon\\
			\epsilon&1
		\end{pmatrix},\\
		\boldsymbol{C}&=\begin{pmatrix}
			\epsilon&(\frac{\sqrt{2}}{2}+\frac{\sqrt{2}}{2}{\bf j})\epsilon&0\\
			0&\frac{\sqrt{2}}{2}-\epsilon&0\\
			0&0&1
		\end{pmatrix}.
\end{align*}
\end{example}

{\bf  Solution.}
Following the proof of \cref{DQRSVD111},
let $\boldsymbol{A}^{(1)}=\boldsymbol{A}, \boldsymbol{B}^{(1)}=\boldsymbol{B}, \boldsymbol{C}^{(1)}=\boldsymbol{C}.$ From the DQGSVD2 of $\begin{pmatrix}
	\boldsymbol{A}\\
	\boldsymbol{C}
\end{pmatrix}$, there exist $\boldsymbol{P}_{1}, \boldsymbol{Q}_{1}$ and $\boldsymbol{V}_{1}$ such that
\begin{align*}
		\boldsymbol{P}_{1}\boldsymbol{A}^{(1)}\boldsymbol{Q}_{1}
		:&=\begin{pmatrix}
			1&-{\bf k}\epsilon&0\\
			-{\bf k}\epsilon&1&0\\
			0&0&1
		\end{pmatrix}
		\begin{pmatrix}
			1&\frac{\sqrt{2}}{2}+(\frac{\sqrt{2}}{2}{\bf k}+1)\epsilon&0\\
			{\bf k}\epsilon&\frac{\sqrt{2}}{2}+(\frac{\sqrt{2}}{2}{\bf k}+1)\epsilon&0\\
			0&0&\epsilon
		\end{pmatrix}\begin{pmatrix}
			1&-\frac{\sqrt{2}}{2}-\epsilon&0\\
			0&1&0\\
			0&0&1
		\end{pmatrix}\\
		&=\begin{pmatrix}
			1&0&0\\
			0&\frac{\sqrt{2}}{2}+\epsilon&0\\
			0&0&\epsilon
		\end{pmatrix},\\
		\boldsymbol{V}_{1}\boldsymbol{C}^{(1)}\boldsymbol{Q}_{1}:&=\begin{pmatrix}
			1&-{\bf j}\epsilon&0\\
			-{\bf j}\epsilon&1&0\\
			0&0&1
		\end{pmatrix}\begin{pmatrix}
			\epsilon&(\frac{\sqrt{2}}{2}+\frac{\sqrt{2}}{2}{\bf j})\epsilon&0\\
			0&\frac{\sqrt{2}}{2}-\epsilon&0\\
			0&0&1
		\end{pmatrix}\begin{pmatrix}
			1&-\frac{\sqrt{2}}{2}-\epsilon&0\\
			0&1&0\\
			0&0&1
		\end{pmatrix}\\
		&=\begin{pmatrix}
			\epsilon&0&0\\
			0&\frac{\sqrt{2}}{2}-\epsilon&0\\
			0&0&1
		\end{pmatrix}.
\end{align*}

Let 
$
	\boldsymbol{P}^{(1)}=\boldsymbol{P}_{1},\
	\boldsymbol{Q}^{(1)}=\boldsymbol{Q}_{1}{\rm diag}(1,\sqrt{2}+2\epsilon,1),\
	\boldsymbol{V}^{(1)}=\boldsymbol{V}_1,\
	\boldsymbol{U}^{(1)}=I_2.
$
Then, $\boldsymbol{A}^{(1)}, \boldsymbol{B}^{(1)}, \boldsymbol{C}^{(1)}$ are updated to be
\begin{subequations}
	\begin{align}
		\label{ex1:A2}
		\boldsymbol{A}^{(2)}&=\boldsymbol{P}^{(1)}\boldsymbol{A}^{(1)}\boldsymbol{Q}^{(1)}=
		\left(\begin{array}{c|cc}
			1&0&0\\ \hline
			0&1+2\sqrt{2}\epsilon&0\\
			0&0&\epsilon
		\end{array}\right),\\
				\label{ex1:B2}
		\boldsymbol{B}^{(2)}&=\boldsymbol{P}^{(1)}\boldsymbol{B}^{(1)}\boldsymbol{U}^{(1)}=\begin{pmatrix}
			1+{\bf i}\epsilon&1+2{\bf j}+{\bf k}\epsilon\\
			\hline
			1+\epsilon&1+(\sqrt{2}-1)\epsilon\\
			\epsilon&1
		\end{pmatrix},\\
        \boldsymbol{C}^{(2)}&=\boldsymbol{V}^{(1)}\boldsymbol{C}^{(1)}\boldsymbol{Q}^{(1)}=\begin{pmatrix}
	\epsilon&0&0\\
	0&1&0\\
	0&0&1
\end{pmatrix}.
	\end{align}
\end{subequations}

Following \eqref{submx}, perform the DQGSVD of the submatrices in \eqref{ex1:A2}-\eqref{ex1:B2} 
\begin{equation*}
	\begin{blockarray}{cccc}
		\begin{block}{(cc|cc)}
			1+2\sqrt{2}\epsilon&0&1+\epsilon&1+(\sqrt{2}-1)\epsilon\\
			0&\epsilon&\epsilon&1\\
		\end{block}
	\end{blockarray},
\end{equation*}
there exist $\boldsymbol{P}_{2}, \boldsymbol{Q}_{2}, \boldsymbol{U}_{2}$ such that
	\begin{align*}
		& \boldsymbol{P}_{2}\begin{pmatrix}
			1+2\sqrt{2}\epsilon&0\\
			0&\epsilon
		\end{pmatrix}\boldsymbol{Q}_{2}\\
		&:=\begin{pmatrix}
			\frac{\sqrt{2}}{2}-\epsilon& -\frac{\sqrt{2}}{2} \\
			0&1
		\end{pmatrix}\begin{pmatrix}
			1+2\sqrt{2}\epsilon&0\\
			0&\epsilon
		\end{pmatrix}\begin{pmatrix}
			1&\epsilon\\
			-\epsilon&1
		\end{pmatrix}\\
		&~=\begin{pmatrix}
			\frac{\sqrt{2}}{2}+\epsilon&0\\
			0&\epsilon
		\end{pmatrix},\\
		& \boldsymbol{P}_{2}\begin{pmatrix}
			1+\epsilon&1+(\sqrt{2}-1)\epsilon\\
			\epsilon&1
		\end{pmatrix}\boldsymbol{U}_{2}\\
		&:=
		\begin{pmatrix}
			\frac{\sqrt{2}}{2}-\epsilon& -\frac{\sqrt{2}}{2} \\
			0&1
		\end{pmatrix}
		\begin{pmatrix}
			1+\epsilon&1+(\sqrt{2}-1)\epsilon\\
			\epsilon&1
		\end{pmatrix}\begin{pmatrix}
			1&\epsilon\\
			-\epsilon&1
		\end{pmatrix}\\
		&~=\begin{pmatrix}
			\frac{\sqrt{2}}{2}-\epsilon&0\\
			0&1
		\end{pmatrix}.
	\end{align*}

Following \eqref{eq:PQUV2}, let
$	\boldsymbol{P}^{(2)}={\rm diag}(1,\boldsymbol{P}_{2}),\
	\boldsymbol{Q}^{(2)}={\rm diag}(1,\boldsymbol{Q}_{2}),\
	\boldsymbol{V}^{(2)}={\rm diag}(1,\boldsymbol{Q}_{2}^{*}),\
	\boldsymbol{U}^{(2)}=\boldsymbol{U}_{2}.  $
Then, $\boldsymbol{A}^{(2)}, \boldsymbol{B}^{(2)}, \boldsymbol{C}^{(2)}$ are updated to be
	\begin{align*}
		&\boldsymbol{A}^{(3)}=\boldsymbol{P}^{(2)}\boldsymbol{A}^{(2)}\boldsymbol{Q}^{(2)}=
		\begin{pmatrix}
			1&0&0\\
			0&\frac{\sqrt{2}}{2}+\epsilon&0\\
			0&0&\epsilon
		\end{pmatrix},\\
		&\boldsymbol{C}^{(3)}=\boldsymbol{V}^{(2)}\boldsymbol{C}^{(2)}\boldsymbol{Q}^{(2)}=\begin{pmatrix}
			\epsilon&0&0\\
			0&1&0\\
			0&0&1
		\end{pmatrix}
			\end{align*}
			\begin{align*}
		&\boldsymbol{B}^{(3)}=\boldsymbol{P}^{(2)}\boldsymbol{B}^{(2)}\boldsymbol{U}^{(2)}=\begin{pmatrix}
			1-(1-{\bf i}+2{\bf j})\epsilon&1+2{\bf j}+(1+{\bf k})\epsilon\\
			\frac{\sqrt{2}}{2}-\epsilon&0\\
			0&1
		\end{pmatrix}.
	\end{align*}

Further take the blocked elementary dual quaternion matrices
	\begin{align*}
		\boldsymbol{P}^{(3)}&=\boldsymbol{P}_{3}=\begin{pmatrix}
			1&-\sqrt{2}+(\sqrt{2}-2-\sqrt{2}{\bf i}+2\sqrt{2}{\bf j})\epsilon&-1-2{\bf j}-(1+{\bf k})\epsilon\\
			0&1&0\\
			0&0&1\\
		\end{pmatrix},\\
		\boldsymbol{Q}^{(3)}&=\boldsymbol{Q}_{3}=\begin{pmatrix}
			1&1-(1-2\sqrt{2}-{\bf i}+2{\bf j})\epsilon&(1+2{\bf j})\epsilon\\
			0&1&0\\
			0&0&1
		\end{pmatrix},  
			\end{align*}	
			and unitary dual quaternion matrices
			\begin{align*}	
		\boldsymbol{U}^{(3)}&=I_{2},\quad
		\boldsymbol{V}^{(3)}=\begin{pmatrix}
			1&-\epsilon&0\\
			\epsilon&1&0\\
			0&0&1
		\end{pmatrix}.
	\end{align*}	
Then, $\boldsymbol{A}^{(3)}, \boldsymbol{B}^{(3)}, \boldsymbol{C}^{(3)}$ are updated to be
\begin{subequations}
	\begin{align*}
		&\boldsymbol{A}^{(4)}=\boldsymbol{P}^{(3)}\boldsymbol{A}^{(3)}\boldsymbol{Q}^{(3)}=\boldsymbol{A}^{(3)},\quad
		\boldsymbol{B}^{(4)}=\boldsymbol{P}^{(3)}\boldsymbol{B}^{(3)}\boldsymbol{U}^{(3)}=\begin{pmatrix}
			0 & 0   \\
			\frac{\sqrt{2}}{2}-\epsilon&0\\
			0&1
		\end{pmatrix},\\
		&\boldsymbol{C}^{(4)}=\boldsymbol{V}^{(3)}\boldsymbol{C}^{(3)}\boldsymbol{Q}^{(3)}=\boldsymbol{C}^{(3)}.
	\end{align*}
\end{subequations}

From \eqref{eq:afbtgm}, let 
\begin{equation*}
	\alpha=\frac{\sqrt{6}}{6}+\frac{5\sqrt{3}}{9}\epsilon,\quad\beta=\frac{\sqrt{2}}{2}-\epsilon,\quad
	\gamma=\frac{\sqrt{3}}{3}+\frac{2\sqrt{6}}{9}\epsilon,
\end{equation*}
and 
\begin{subequations}
	\begin{align*}
		\boldsymbol{P}^{(4)}&=I_{3},\quad
		\boldsymbol{Q}^{(4)}=\begin{pmatrix}
			1&0&0\\
			0&\frac{\sqrt{3}}{3}+\frac{2\sqrt{6}}{9}\epsilon&0\\
			0&0&1
		\end{pmatrix},\\
		\boldsymbol{U}^{(4)}&=I_{2},\quad
		\boldsymbol{V}^{(4)}=I_{3}.
	\end{align*}
\end{subequations}
Then we have 
\begin{subequations}
	\begin{align*}
		&\boldsymbol{P}^{(4)}\boldsymbol{A}^{(4)}\boldsymbol{Q}^{(4)}=\begin{pmatrix}
			1&0&0\\
			0&\frac{\sqrt{6}}{6}+\frac{5\sqrt{3}}{9}\epsilon&0\\
			0&0&\epsilon
		\end{pmatrix},\
		\boldsymbol{P}^{(4)}\boldsymbol{B}^{(4)}\boldsymbol{U}^{(4)}=\begin{pmatrix}
			0 & 0 \\
			\frac{\sqrt{2}}{2}-\epsilon&0\\
			0&1
		\end{pmatrix},\\
		&\boldsymbol{V}^{(4)}\boldsymbol{C}^{(4)}\boldsymbol{Q}^{(4)}=\begin{pmatrix}
			\epsilon&0&0\\
			0&\frac{\sqrt{3}}{3}+\frac{2\sqrt{6}}{9}\epsilon&0\\
			0&0&1
		\end{pmatrix}.
	\end{align*}
\end{subequations}

In summary, by taking
\begin{align*}
\boldsymbol{P}&=\boldsymbol{P}^{(4)}\boldsymbol{P}^{(3)}\boldsymbol{P}^{(2)}\boldsymbol{P}^{(1)}\\
&=\begin{pmatrix}
	1+{\bf k}\epsilon &  -1+(1-{\bf i}+2{\bf j}-{\bf k})\epsilon &  -2{\bf j}-(2-\sqrt{2}-{\bf i}+2{\bf j}+{\bf k})\epsilon  \\
	-\frac{\sqrt{2}}{2}{\bf k}\epsilon &  \frac{\sqrt{2}}{2} -\epsilon &  -\frac{\sqrt{2}}{2}  \\
	0 & 0& 1
	\end{pmatrix}, \\ \boldsymbol{Q}&=\boldsymbol{Q}^{(1)}\boldsymbol{Q}^{(2)}\boldsymbol{Q}^{(3)}\boldsymbol{Q}^{(4)}=\begin{pmatrix}
	1 & -(\frac{\sqrt{3}}{3}-\frac{\sqrt{3}}{3}{\bf i}+\frac{2\sqrt{3}}{3}{\bf j})\epsilon  & 2{\bf j}\epsilon  \\
	0 & \frac{\sqrt{6}}{3}+\frac{10\sqrt{3}}{9}\epsilon  &  \sqrt{2}\epsilon \\
	0 &  -\frac{\sqrt{3}}{3}\epsilon &  1
\end{pmatrix}, \\
	\boldsymbol{U}&=\boldsymbol{U}^{(1)}\boldsymbol{U}^{(2)}\boldsymbol{U}^{(3)}\boldsymbol{U}^{(4)}=\begin{pmatrix}
		1&\epsilon\\
		-\epsilon&1
	\end{pmatrix}, \\  	\boldsymbol{V}&=\boldsymbol{V}^{(4)}\boldsymbol{V}^{(3)}\boldsymbol{V}^{(2)}\boldsymbol{V}^{(1)}=\begin{pmatrix}
	1 & -(1+{\bf j})\epsilon  &  0\\
	(1-{\bf j})\epsilon & 1 & -\epsilon\\
	0 & \epsilon  & 1
	\end{pmatrix},
 \end{align*}
we can find the restricted SVD of $(\boldsymbol{A},\boldsymbol{B},\boldsymbol{C})$ to be
\begin{subequations}
	\begin{align*}
		&\boldsymbol{P}\boldsymbol{A}\boldsymbol{Q}=\begin{pmatrix}
			1&0&0\\
			0&\frac{\sqrt{6}}{6}+\frac{5\sqrt{3}}{9}\epsilon&0\\
			0&0&\epsilon
		\end{pmatrix},\
		\boldsymbol{P}\boldsymbol{B}\boldsymbol{U}=\begin{pmatrix}
			0&0\\
			\frac{\sqrt{2}}{2}-\epsilon&0\\
			0&1
		\end{pmatrix},\\
		&\boldsymbol{V}\boldsymbol{C}\boldsymbol{Q}=\begin{pmatrix}
			\epsilon&0&0\\
			0&\frac{\sqrt{3}}{3}+\frac{2\sqrt{6}}{9}\epsilon&0\\
			0&0&1
		\end{pmatrix}.
	\end{align*}
\end{subequations}

\begin{example}
Find the product-product SVD of the following dual quaternion triplet $(\boldsymbol{A},\boldsymbol{B},\boldsymbol{C})$, where
\begin{equation*}
	\begin{aligned}
		\boldsymbol{A}=\begin{pmatrix}
			\frac{1}{2}+\epsilon&-\frac{1}{2}{\bf i}\epsilon\\
			\frac{1}{2}{\bf k}\epsilon&\epsilon
		\end{pmatrix},\quad
		\boldsymbol{B}=\begin{pmatrix}
			1&{\bf i}\epsilon\\
			{\bf k}\epsilon&\frac{1}{2}
		\end{pmatrix},\quad
		\boldsymbol{C}=\begin{pmatrix}
			1&{\bf k}\epsilon\\
			{\bf j}\epsilon&{\bf i}\epsilon
		\end{pmatrix}.
	\end{aligned}
\end{equation*}
\end{example}

{\bf Solution.}
Following the proof of \cref{ThmDQPPRSVD}. According to \cref{PrePPSVD}, there exists 
$	\boldsymbol{U}_{1}=\begin{pmatrix}
		1&{\bf i}\epsilon\\
		{\bf i}\epsilon&1
	\end{pmatrix}$
such that
\begin{equation*}
		\boldsymbol{A}\boldsymbol{U}_{1}=\begin{pmatrix}
			\frac{1}{2}+\epsilon&0\\
			\frac{1}{2}{\bf k}\epsilon&\epsilon
		\end{pmatrix},\quad
		\boldsymbol{U}_{1}^{*}\boldsymbol{B}=\begin{pmatrix}
			1&\frac{1}{2}{\bf i}\epsilon\\
			({\bf k}-{\bf i})\epsilon&\frac{1}{2}
		\end{pmatrix}.
\end{equation*}
Then we can find a unitary dual quaternion matrix $\boldsymbol{U}_{2}=\begin{pmatrix}
	1&-\frac{1}{2}{\bf i}\epsilon\\
	-\frac{1}{2}{\bf i}\epsilon&1
\end{pmatrix}$ such that
\begin{equation*}
		\widetilde{\boldsymbol{B}}=\boldsymbol{U}_{1}^{*}\boldsymbol{B}\boldsymbol{U}_{2}
		=\begin{pmatrix}
			1&0\\
			({\bf k}-\frac{5}{4}{\bf i})\epsilon&\frac{1}{2}
		\end{pmatrix},\quad
		\boldsymbol{U}_{2}^{*}\boldsymbol{C}=\begin{pmatrix}
			1&{\bf k}\epsilon\\
			(\frac{1}{2}{\bf i}+{\bf j})\epsilon&{\bf i}\epsilon
		\end{pmatrix}.
\end{equation*}

Further from \cref{PrePPSVD}, we find $\boldsymbol{U}_{3}=\begin{pmatrix}
	1&-{\bf k}\epsilon\\
	-{\bf k}\epsilon&1
\end{pmatrix}$ 
such that
\begin{equation*}
	\widetilde{\boldsymbol{C}}=\boldsymbol{U}_{2}^{*}\boldsymbol{C}\boldsymbol{U}_{3}=\begin{pmatrix}
		1&0\\
		(\frac{1}{2}{\bf i}+{\bf j})\epsilon&{\bf i}\epsilon
	\end{pmatrix}.
\end{equation*}
Let 
$
	\boldsymbol{G}_{1}^{-1}=\begin{pmatrix}
		1&0\\
		-(\frac{1}{2}{\bf i}+{\bf j})\epsilon&1
	\end{pmatrix}.
$
Then
\begin{equation*}
	\begin{aligned}
		&\boldsymbol{G}_{1}^{-1}\widetilde{\boldsymbol{C}}=\begin{pmatrix}
			1&0\\
			-(\frac{1}{2}{\bf i}+{\bf j})\epsilon&1
		\end{pmatrix}\begin{pmatrix}
			1&0\\
			(\frac{1}{2}{\bf i}+{\bf j})\epsilon&{\bf i}\epsilon
		\end{pmatrix}=\begin{pmatrix}
			1&0\\
			0&{\bf i}\epsilon
		\end{pmatrix},\\
		&\widetilde{\boldsymbol{B}}\boldsymbol{G}_{1}=\begin{pmatrix}
			1&0\\
			({\bf k}-\frac{5}{4}{\bf i})\epsilon&\frac{1}{2}
		\end{pmatrix}\begin{pmatrix}
			1&0\\
			(\frac{1}{2}{\bf i}+{\bf j})\epsilon&1
		\end{pmatrix}=\begin{pmatrix}
			1&0\\
			({\bf k}+{\bf j}-\frac{3}{4}{\bf i})\epsilon&\frac{1}{2}
		\end{pmatrix}.
	\end{aligned}
\end{equation*}
Further let 
$
	\boldsymbol{G}_{2}^{-1}=\begin{pmatrix}
		1&0\\
		-({\bf k}+{\bf j}-\frac{3}{4}{\bf i})\epsilon&1
	\end{pmatrix}.
$
Then
	\begin{align*}
		&\boldsymbol{G}_{2}^{-1}\widetilde{\boldsymbol{B}}\boldsymbol{G}_{1}=\begin{pmatrix}
			1&0\\
			-({\bf k}+{\bf j}-\frac{3}{4}{\bf i})\epsilon&1
		\end{pmatrix}\begin{pmatrix}
			1&0\\
			({\bf k}+{\bf j}-\frac{3}{4}{\bf i})\epsilon&\frac{1}{2}
		\end{pmatrix}=\begin{pmatrix}
			1&0\\
			0&\frac{1}{2}
		\end{pmatrix},\\
		\label{eq:maxA}
		&\widetilde{\boldsymbol{A}}=\boldsymbol{A}\boldsymbol{U}_{1}\boldsymbol{G}_{2}=\begin{pmatrix}
			\frac{1}{2}+\epsilon&0\\
			\frac{1}{2}{\bf k}\epsilon&\epsilon
		\end{pmatrix}\begin{pmatrix}
			1&0\\
			({\bf k}+{\bf j}-\frac{3}{4}{\bf i})\epsilon&1
		\end{pmatrix}=\begin{pmatrix}
			\frac{1}{2}+\epsilon&0\\
			\frac{1}{2}{\bf k}\epsilon&\epsilon
		\end{pmatrix}.
	\end{align*}

Construct $\boldsymbol{U}_{4}=\begin{pmatrix}
	1&-{\bf k}\epsilon\\
	-{\bf k}\epsilon&1
\end{pmatrix}$. 
Then we have
\begin{equation*}
	\boldsymbol{U}_{4}\widetilde{\boldsymbol{A}}=\begin{pmatrix}
		1&-{\bf k}\epsilon\\
		-{\bf k}\epsilon&1
	\end{pmatrix}\begin{pmatrix}
		\frac{1}{2}+\epsilon&0\\
		\frac{1}{2}{\bf k}\epsilon&\epsilon
	\end{pmatrix}=\begin{pmatrix}
		\frac{1}{2}+\epsilon&0\\
		0&\epsilon
	\end{pmatrix}.
\end{equation*}
Let $\widetilde{\boldsymbol{P}}=\begin{pmatrix}
	2-4\epsilon & 0 \\
	0 & 1
\end{pmatrix}$, $\hat{\boldsymbol{P}}=\begin{pmatrix}
	1&0\\
	0&2
\end{pmatrix}$. 
Then we have
\begin{align*}
	\boldsymbol{U}_{4}\widetilde{\boldsymbol{A}}\widetilde{\boldsymbol{P}}&=
	\begin{pmatrix}
		\frac{1}{2}+\epsilon&0\\
		0&\epsilon
	\end{pmatrix}\begin{pmatrix}
	2-4\epsilon & 0 \\
	0 & 1
	\end{pmatrix}=\begin{pmatrix}
	1&0\\
	0&\epsilon
	\end{pmatrix}.	\\
	\widetilde{\boldsymbol{P}}^{-1}\boldsymbol{G}_{2}^{-1}\widetilde{\boldsymbol{B}}\boldsymbol{G}_{1}\widetilde{\boldsymbol{P}}\hat{\boldsymbol{P}}&=\begin{pmatrix}
		\frac{1}{2}+\epsilon&0\\
		0&1
	\end{pmatrix}\begin{pmatrix}
	1&0\\
	0&\frac{1}{2}
	\end{pmatrix}\begin{pmatrix}
	2-4\epsilon & 0 \\
	0 & 1
	\end{pmatrix}\begin{pmatrix}
	1 & 0 \\
	0 & 2
	\end{pmatrix}=\begin{pmatrix}
	1&0\\
	0&1
	\end{pmatrix}. \\
	\hat{\boldsymbol{P}}^{-1}\widetilde{\boldsymbol{P}}^{-1}\boldsymbol{G}_{1}^{-1}\widetilde{\boldsymbol{C}}&=\begin{pmatrix}
		1&0\\
		0&\frac{1}{2}
	\end{pmatrix}\begin{pmatrix}
	\frac{1}{2}+\epsilon&0\\
	0&1
	\end{pmatrix}\begin{pmatrix}
	1&0\\
	0&{\bf i}\epsilon
	\end{pmatrix}=\begin{pmatrix}
	\frac{1}{2}+\epsilon&0\\
	0&\frac{1}{2}{\bf i}\epsilon
	\end{pmatrix}.
\end{align*}

Let 
$\boldsymbol{U}^{*}=\boldsymbol{U}_{4},\  \boldsymbol{V}=\boldsymbol{U}_{3}.$
We eventually obtain DQSVD of the matrix $\boldsymbol{ABC}$ to be
\begin{equation*}
	\boldsymbol{U}^{*}(\boldsymbol{ABC})\boldsymbol{V}=\begin{pmatrix}
		1&-{\bf k}\epsilon\\
		-{\bf k}\epsilon&1
	\end{pmatrix}
	\begin{pmatrix}
	    \frac{1}{2}+\epsilon &  \frac{1}{2}{\bf k}\epsilon  \\
	     \frac{1}{2}{\bf k}\epsilon  &  0
	\end{pmatrix}
	\begin{pmatrix}
		1&-{\bf k}\epsilon\\
		-{\bf k}\epsilon&1
	\end{pmatrix}
	=\begin{pmatrix}
		\frac{1}{2}+\epsilon&0\\
		0&0
	\end{pmatrix}.
\end{equation*}
In summary, by taking
\begin{equation*}
	\boldsymbol{P}=\boldsymbol{U}_{1}\boldsymbol{G}_{2}\widetilde{\boldsymbol{P}}
=\begin{pmatrix}
	2-4\epsilon&{\bf i}\epsilon\\
	(\frac{1}{2}{\bf i}+2{\bf j}+2{\bf k})\epsilon&1
	\end{pmatrix}, \   \  
	\boldsymbol{Q}=\boldsymbol{U}_{2}\boldsymbol{G}_{1}\widetilde{\boldsymbol{P}}\hat{\boldsymbol{P}}
=\begin{pmatrix}
	2-4\epsilon & -\frac{1}{2}{\bf i}\epsilon \\
	2{\bf j}\epsilon & 2
	\end{pmatrix},
\end{equation*}
we have 
\begin{align*}
	&\boldsymbol{U}^{*}\boldsymbol{A}\boldsymbol{P}=\begin{pmatrix}
		1&0\\
		0&\epsilon
	\end{pmatrix}, \boldsymbol{P}^{-1}\boldsymbol{B}\boldsymbol{Q}=\begin{pmatrix}
	1&0\\
	0&1
	\end{pmatrix},\
	\boldsymbol{Q}^{-1}\boldsymbol{C}\boldsymbol{V}=\begin{pmatrix}
		\frac{1}{2}+\epsilon&0\\
		0&\frac{1}{2}{\bf i}\epsilon
	\end{pmatrix}.
\end{align*}

\section{Conclusions}\label{sec:concl}

In this paper, we have investigated the restricted singular value decomposition and product-product singular value decomposition of dual quaternion matrix triplets $(\boldsymbol{A},\boldsymbol{B},\boldsymbol{C})$. The RSVD and PPSVD of dual quaternion matrices differ from those over the complex field due to the unique properties of the purely dual part of dual quaternion matrices. Two examples have been presented to illustrate the mentioned decompositions. This paper solely focuses on the feasibility of the decompositions; meanwhile, certain geometric interpretations, algebraic issues, and uniqueness properties merit further discussion in future research.









\section*{Acknowledgments}
The work is supported by the Fundamental Research Funds for the Central Universities under grant 2024ZDPYCH1004.

\end{document}